\documentclass[11pt]{amsart}
\usepackage[margin=1.2in]{geometry}
\usepackage{amsmath}
\usepackage{amsthm}
\usepackage{color}
\usepackage{amsfonts}
\usepackage{amssymb}
\usepackage{cancel}
\usepackage{amscd}
\numberwithin{equation}{section}
\newtheorem{theorem}{Theorem}[section]
\newtheorem{cor}[theorem]{Corollary}

\newtheorem{proposition}[theorem]{Proposition}
\newtheorem{prop}[theorem]{Proposition}
\newtheorem{lemma}[theorem]{Lemma}
\theoremstyle{definition}
\newtheorem{defi}{Definition}[section]
\newtheorem{rem}{Remark}[section]

\newtheorem{example}[defi]{Example}
\newtheorem*{thma}{Theorem}

\newcommand\half{\frac{1}{2}}

\newcommand\ov{\overline}

\newcommand\be{\beta}

\newcommand\g{\mathfrak g}
\newcommand\ga{\widehat{\mathfrak g}}
\newcommand\h{\mathfrak h}

\newcommand\D{\Delta}
\renewcommand\l{\lambda}
\newcommand\Dp{\Delta^+}

\renewcommand\d{\delta}

\renewcommand\a{\alpha}
\renewcommand\aa{\mathfrak a}

\renewcommand\k{\mathfrak k}

\newcommand{\Z}{\mathbb Z}

\renewcommand\i{{\mathfrak i}}

\newcommand\ganz{\mathbb Z}

\renewcommand\L{\Lambda}

\renewcommand\aa{\mathfrak a}

\newcommand\C{\mathbb C}

\newcommand\p{\mathfrak p}

\newcommand{\vac}{{\bf 1}}
\newcommand{\bea}{\begin{eqnarray}}
\newcommand{\eea}{\end{eqnarray}}
\begin{document}
\title[Affine conformal embeddings and applications]{On the classification  of non-equal rank affine conformal embeddings and applications}
\author[Adamovi\'c, Kac, M\"oseneder, Papi, Per\v{s}e]{Dra{\v z}en~Adamovi\'c}
\author[]{Victor~G. Kac}
\author[]{Pierluigi M\"oseneder Frajria}
\author[]{Paolo  Papi}
\author[]{Ozren  Per\v{s}e}
\begin{abstract} We complete the classification of conformal embeddings of   a maximally reductive subalgebra $\k$ into a simple Lie algebra $\g$ at non-integrable non-critical levels $k$ by dealing with the case when $\k$ has rank less than that of $\g$. 
 We describe some remarkable instances of decomposition of the vertex algebra $V_{k}(\g)$ as a module for the vertex subalgebra generated by $\k$. We discuss decompositions of conformal embeddings and constructions of new affine  Howe dual pairs at negative levels.
 In particular, we study an example of  conformal embeddings $A_1 \times A_1    \hookrightarrow  C_3$ at level $k=-1/2$, and obtain explicit branching rules by applying certain $q$-series identity. In the analysis of conformal embedding $A_1 \times D_4   \hookrightarrow  C_8$ at level $k=-1/2$ we detect subsingular vectors which do not appear in  the branching rules of the  classical Howe dual  pairs.
\end{abstract}
\maketitle
\centerline{\it To the memory of Bertram Kostant 5/24/1928--2/2/2017}
\section{Introduction}
Let $\g$ be a semisimple finite-dimensional complex Lie algebra and $\k$  a reductive subalgebra of $\g$. The embedding $\k\hookrightarrow \g$ is  called conformal    if the central charge of the 
Sugawara construction for
the affinization $\ga$,  acting on an  integrable $\ga$--module of level $k$, equals that for $\widehat{\k}$. Then necessarily $k=1$ \cite{AGO}.
Maximal conformal embeddings were classified
in \cite{SW},  \cite{AGO}, and related decompositions can be found in \cite{KWW}, \cite{KS}, \cite{CKMP}.
 In the  vertex algebra framework the definition can be rephrased as follows: the simple affine vertex  algebras $V_1(\g)$ and the vertex subalgebra  generated by $\{x_{(-1)}\vac\mid x\in\k\}$ have the same
Sugawara conformal vector.\par
Let us denote by $\widetilde V(k,\k)$  the vertex subalgebra  of $V_k(\g)$  generated by $\{x_{(-1)}\vac\mid x\in\k\}$. In \cite{AKMPP1} we generalized the previous situation to study when the simple affine vertex  algebra $V_k(\g)$ and its subalgebra  $\widetilde V(k,\k)$ have the same
Sugawara conformal vector  for some  non-critical level $k$, not necessarily $1$, assuming  $\k$ to be  an equal rank reductive subalgebra. We also considered the problem of 
 providing the explicit decomposition of $V_k(\g)$ regarded as a $\widetilde V(k,\k)$--module.\par
 The present paper is divided into two parts.  In the first part (Sections 3--5) we  deal with the classification problem in the non-equal rank case: in particular, we give the complete classification of conformal embeddings 
 when $\k$ is a maximal non-equal rank  semisimple  subalgebra of $\g$.
  In the second part (Sections 7--11) we discuss  some instances of the decomposition problem that have  interesting applications, such as a representation theoretic interpretation of an $\eta$-function identity and  the emergence of   new Howe dual pairs. 
  Section 6 combines the results from Sections 3--5 and those of  \cite{AKMPP1}  to obtain  the classification of conformal embeddings of maximally reductive subalgebras (cf. Definition \ref{mr}) in simple Lie algebras: see Theorem \ref{global}.\par
  
 We should mention that the   some of the affine vertex algebras  occurring in our    analysis of conformal embeddings, have appeared recently in various mathematics and physics papers
on simple affine vertex algebras associated with the Deligne exceptional series at levels $k =-h^{\vee}/ 6 -1$ \cite{ArK}, \cite{ArM-I}, \cite{ArM-II}, \cite{Ga}.

  %
  %
   \subsection{Classification of conformal embeddings}    We now discuss the methods employed in our solution of the classification problem. Our main tool is a criterion given in \cite{A} for conformal embeddings (see Subsection \ref{APC}), referred to in the following as the {\it AP-criterion}. As explained in Subsection \ref{classDyn}, the classification of maximally reductive subalgebras reduces  to Dynkin's classification of maximal semisimple subalgebras of a simple Lie algebra. This classification splits these subalgebras in some classes. For each of these classes, we develop methods to enforce the AP-criterion.
\par
In Section 3, we discuss 
  the conformal embeddings of $\widetilde V(k,\k)$ in  $V_k(\g)$ with  $\g=so(V), sp(V),$ $ sl(V)$, and $V$ irreducible as a representation of $\k$. Each of these cases requires similar but not uniform approaches. 
  In the case of $\g=so(V)$, we use Kostant's theory of pairs of Lie type \cite{Kost}.  Kostant found a  condition in terms of the Clifford algebra $Cl(V)$ for $\k\oplus V$ to have a Lie algebra structure with natural  properties.
Reformulating the Symmetric Space Theorem \cite{GNO} in terms of  pairs of Lie type, we find in Proposition \ref{strong} a very strong condition for the existence of conformal embeddings at level $k\ne 1$. It turns out  that  all but two cases are ruled  out (cf. Proposition \ref{ce}):


\begin{theorem}\label{class-intr}
Assume that $\widetilde V(k,\k)$ embeds conformally in $V_k(so(V))$ such that  $k\not\in \ganz_+$ and $V$ is an irreducible $\k$--module.
Then $k=-2$, and  we are in one of  the following cases 
 \begin{itemize}
\item $\k= B_3$, $ V= L_{B_3} (\omega_3)$, conformal embedding   of $V_{-2} (B_3)$ into $V_{-2} (D_4) = V_{-2} (so(V))$.
\item  $\k= G_2$, $V= L_{G_2} (\omega_1)$, conformal embedding    of  $V_{-2} (G_2)$ into $V_{-2} (B_3) = V_{-2} (so(V))$.
 \end{itemize} 
\end{theorem}

The same ideas are employed in the case of embeddings in $sp(V)$, where  pairs of Lie type are  substituted by Lie superalgebras of Riemannian type \cite{Kostsuper} and the Clifford algebra of $V$ is substituted by the symmetric algebra $S(V)$. The final outcome is contained in Proposition \ref{cesp}; note that in this case we have to use genuine ``super'' techniques, as well as 
Kac's classification of simple Lie superalgebras. 
The embeddings in $sl(V)$ are dealt with by adapting Kostant's ideas to the algebra $End(V)$ and the corresponding classification appears in Proposition \ref{cesl}.     \par
In Section 4, we deal with conformal embeddings of $\widetilde V(k,\k)$ in  $V_k(\g)$ with $\g$ classical and $\k$ semisimple non-simple. In this case the verification of the conditions 
of the AP-criterion is performed by using Classical Invariant Theory (cf. Theorem \ref{thm1.1}). It is worthwhile to remark that, as a consequence of our analysis, we are able to provide examples 
of conformal embeddings $\widetilde V(k,\k)\subset V_k(\g)$ with $V_k(\g)$ non semisimple as a $\widetilde V(k,\k)$--module: see Example \ref{ex-non-simp} and Theorem \ref{thm-simplicity}.\par
Section 5 is devoted to a direct analysis of conformal embeddings   in  $V_k(\g)$ with $\g$ of exceptional type.\par

 \subsection{Decomposition of embeddings}
 The decomposition problem for conformal embeddings was  studied in our previous paper  \cite{AKMPP1} for equal rank affine embeddings,  and in \cite{AKMPP2} for embeddings of affine vertex algebras into some $\mathcal{W}$--algebras. Quite surprisingly, in handling decompositions of conformal embeddings for non-equal rank subalgebras we found  completely new phenomena: 
hence  new ideas for a general approach will be required; here we analyze some special cases.
 Even in integrable cases such decompositions are related with some non-trivial results on symmetric spaces \cite{CKMP}, the theory of simple-current extensions \cite{KMPX}, and the  affine tensor categories and rank-level duality \cite{Ost}. In non-integrable cases, the decomposition problem  is very difficult, since most of the tools used in integrable cases do not apply.
 One reason, as noted above,  is that  decompositions  in   non-integrable cases need not to be  semisimple.
 In this paper we study in detail the decomposition of the Weyl vertex algebra $M_{(m)}$ as a $V^ {-m/2}(sl(2)) \otimes   V^{-2}(so(m))$--module, which exhibits an infinite-dimensional  generalization  of  classical   results due to Howe \cite[Section 4]{H-Tams}.  As a byproduct, we  get the realization of the simple affine vertex algebra $V_{-2} (A_3)$ (investigated in  \cite{ArM-II}) inside of the Weyl vertex algebra $M_{(6)}$.   Theorem  \ref{thm-simplicity} implies the following result.
 
 \begin{proposition}
 There exists a non-trivial homomorphism $\Phi : V_{-2} (A_3) \rightarrow M_{(6)}$. 
 \end{proposition}
 
 We hope that this realization can be used to verify  the fusion-rules conjecture   for $V_{-2} (A_3) $--modules from the category $KL_{-2}$  presented in \cite{ArM-II}.
 
 Moreover, in Section \ref{analysis-II} we identify all singular vectors in $M_{(m)}$ which correspond to singular  vectors obtained  using classical invariant theory. The special  case $m=8$ is further studied in Section \ref{m8}, where we show that $M_{(8)}$ has subsingular vectors which do not appear in the classical case: see Proposition \ref{prop-new}.

  \subsection{A connection with Howe dual pairs}
   The vertex subalgebra $\widetilde V(k,\k)$ of $V_k(\g)$ usually has the form  $\widetilde V(k,\k_1) \otimes \widetilde V(k,\k_2) $ for certain simple Lie algebras $\k_i$, and it is natural to consider the commutant (=coset) vertex algebra $\mbox{Com} (\widetilde V(k,\k_i), {V}_k(\g))$.  The determination of commutants is  slightly easier   than the problem of  finding explicit decompositions, nevertheless it gives relevant information on the structure of embeddings. If  
   $$\mbox{Com} (\widetilde V(k,\k_1), {V}_k(\g)) = \widetilde V(k,\k_2), \quad  \mbox{Com} (\widetilde V(k,\k_2), {V}_k(\g)) = \widetilde V(k,\k_1), $$ 
   we say that  $\widetilde V(k,\k_1)$ and  $\widetilde V(k,\k_2)$ are   an affine Howe dual pair inside of $V_k(\g)$ .
   
   In the vertex algebra setting the commutant problem and  Howe dual pairs were  extensively studied by A. Linshaw and collaborators (see  \cite{LS}, \cite{LSS} and reference therein). We have noticed that the results from \cite{LSS}  can be applied to  conformal embeddings from  our paper.   In  Proposition  \ref{g-1}   we prove the following.
 \begin{theorem} For $m\ge 3$, we have 
\begin{equation*}Com(V_{-m } (sl(2)), M_{(2 m)} ) = \widetilde V_{-2} (so(2 m)), \end{equation*}
where  $\widetilde V_{-2} (so(2 m))$ is a certain quotient of $V^{-2} (so(2m))$.

The vertex algebra $V^{-2} (so(2m))$ is  simple  if and only if $m=3$.
\end{theorem}
   
   Then, by combining methods from \cite{LSS} and   \cite{AKMPP1}, \cite{AKMPP2} we are able to construct  new examples of affine Howe dual pairs in   Corollary \ref{cor-11}.
   
 \begin{theorem}   The following pairs of vertex algebras are affine Howe dual pairs:
\begin{itemize}
\item[(1)] $V_{-m} (sl(2))$ and $\widetilde V_{-2} (D_m)$ inside  $V_{-1/2} (C_{2m})$ for $m \ge 3$.
\item[(2)] $V_{-m} (sl(2))$ and $\overline V_{-2} (A_{m-1})$ inside  $V_{-1} (A_{2m-1})$ for $m \ge 5$.
\end{itemize}
\end{theorem}

 The case $m=4$ is related   with a recent physics conjecture of D. Gaiotto \cite{Ga} (cf. Remark \ref{rem-81}).

  \subsection{Examples of branching rules }
   Even in the cases when   the decomposition is semi-simple,  $V_k(\g)$ is not a simple current extension of $\widetilde V(k,\k)$, and studying such extensions is a very difficult problem in vertex algebra theory.  Nevertheless, in some cases, explicit decompositions can be obtained by using combinatorial/number theoretic methods and applying certain $q$-series identities.  In the present paper we present a decomposition of the vertex algebra $M_{(3)}$ as a $V_{-4}(sl(2)) \otimes V_{-3/2} (sl(2))$--module. We identify all singular vectors in $M_{(3)}$ and  the characters of  certain irreducible $V_{-4}(sl(2)) \otimes V_{-3/2} (sl(2))$--modules. 
   Then, the decomposition of $M_{(3)}$ is obtained in Theorem \ref{m3-dec} as a consequence of the  $q$-series identity from \cite[Example 5.2]{KW-1994}.
\begin{theorem}   $M_{(3)}$ is a completely reducible $V_{-3/2} (sl(2)) \otimes V_{-4}(sl(2))$--module and the following decomposition holds
\bea  && M_{(3) }=  \bigoplus_{\ell =0} ^{\infty}\left( L_{\widehat{sl(2)}} (- (\frac{3}{2} + \ell  ) \Lambda_0 +  \ell  \Lambda_1) \bigotimes    L_{\widehat{sl(2)}} (- (4 +  2 \ell  ) \Lambda_0 +    2 \ell \Lambda_1)   \label{m3-1} \right).\eea
\end{theorem}

\vskip 5mm
 {\bf Acknowledgment.} This work was done in part during the authors' stay at Erwin Schr\"odin\-ger
Institute in Vienna  (January  2017).  D.A. and O. P. are  partially supported by the Croatian Science Foundation under the project 2634 and by the
QuantiXLie Centre of Excellence, a project cofinanced
by the Croatian Government and European Union
through the European Regional Development Fund - the
Competitiveness and Cohesion Operational Programme
(KK.01.1.1.01).  We thank the referee for his/her careful reading of the paper.
\vskip 5mm


\section{Setup}
\subsection{AP-criterion}\label{APC} Let $\g$ be  a simple Lie algebra. Let $\h$ be a Cartan subalgebra, $\D$ the 
$(\g,\h)$--root system, $\Dp$  a set of positive roots and $\rho$ the corresponding Weyl vector. Let $(\cdot,\cdot)$ denote the   normalized bilinear invariant  form (i.e., $(\a,\a)=2$ for any long root). The dual Coxeter number is denoted by $h^\vee$. This is half the eigenvalue of the Casimir element corresponding to $(\cdot,\cdot)$ when acting on $\g$.
We shall write  $L_\g(\eta)$ to denote the irreducible highest weight $\g$--module of highest weight $\eta$; when clear from the context, we write simply $L(\eta)$. We let  $V^{k}(\g)$ be the universal affine vertex algebra of level $k$ and, if $k+h^\vee\ne0$, we denote by  $V_{k}(\g)$ its simple quotient. We let $L_{\widehat \g}(\Lambda)$ denote the irreducible highest weight $V^k(\g)$--module as well as, when the action pushes down, the corresponding $V_k(\g)$--module.
 The notation $\widetilde L_{\widehat \g}(\Lambda)$ usually denotes  a  highest weight $V^{k}(\g)$--module, not necessary simple, with highest weight $\Lambda$. Similarly $\widetilde V_k(\g)$ or  $\overline V_k(\g)$ usually denote a quotient of $V^k(\g)$ which are  possibly different than $V_k(\g)$.

Assume that $\k$ is a semisimple subalgebra of $\g$. Then $\k$ decomposes as 
\begin{equation*}\label{decompg0}
\k=\k_1\oplus\cdots\oplus \k_t.
\end{equation*}
where 
$\k_1,\ldots \k_t$ are the simple ideals of $\k$. 
Let $\p$ be  the orthocomplement of $\k$ w.r.t to $(\cdot, \cdot)$ and let 
$$\p=\bigoplus_{i=1}^s\bigotimes_{j=1}^{t} L_\k(\mu_i^j)$$ be its  decomposition  as a $\k$--module.   \par Let  $( \cdot, \cdot)_j$ denote the normalized invariant
bilinear form on $\k_j$.  We denote by $\widetilde V(k,\k)$ the vertex subalgebra of $V_k(\g)$ generated by $\{x_{(-1)}\vac\mid x\in\k\}$. Note that $\widetilde V(k,\g)$ is an affine vertex algebra, more precisely it is a quotient of $\otimes V^{k_j}(\k_j)$, with the levels $k_j$ determined by $k$ and the ratio between $(\cdot,\cdot)$ and $(\cdot,\cdot)_j$.

\begin{thma}(AP-criterion) \cite{A} {\it
 $\widetilde V(k,\k)$ is
conformally embedded in $V_k({\g})$ if and only if
\bea\label{numcheck} && \sum_{j=0}^t\frac{ ( \mu_{i}^j , \mu _{i}^j + 2 \rho_0^j ) _j}{ 2 (k_j + h_j
^{\vee} )} = 1  \eea
for any $i=1,\ldots,s$.}
\end{thma}

\begin{rem}
We note that  the AP-criterion allows  conformal embeddings in non-simple vertex algebras. Let $\omega_{\g}$ (resp. $\omega_{\k}$) be the Sugawara Virasoro vector in $V^{k}(\g)$ (resp. ${V}^{k}(\k) $). It was proved in \cite{A} that (\ref{numcheck}) implies that $\omega_{\k} -\omega_{\g} $ belongs to the maximal ideal in $V^k({\g})$. So we automatically have conformal embedding of  $ {\mathcal V}_{k}(\k)$(= certain quotient of $\widetilde V(k,\k)$)  in the vertex algebra $$\frac{V^{k}(\g)}{ V^{k}(\g) .  (\omega_{\k} -\omega_{\g}) },$$
which doesn't need to be simple. In  present paper we identify a  family of conformal embeddings in non-simple vertex algebras (see  Corollaries  \ref{cor-branching} and  \ref{cor-non-simp-2}).
\end{rem}
We reformulate the criterion highlighting  the dependence from the choice of the form $(\cdot,\cdot)$. As invariant symmetric form on $\k$ we choose  $(\cdot,\cdot)_{|\k\times \k}$. Fix an orthonormal basis $\{X_i\}$ of $\k_j$ and let $C_{\k_j}= \sum_i X_i^2$ be the corresponding Casimir operator. Let $2 g_j$ be the eigenvalue for the action of $C_{\k_j}$ on $\k_j$ and $\gamma^j_i$ the eigenvalue of the action of $C_{\k_j}$ on $L(\mu^j_i)$. Then  $ \widetilde{V}_{k}(\k)$ is
conformally embedded in $V_k({\g})$ if and only if
\bea\label{numcheck2} && \sum_{j=0}^t\frac{ \gamma^j_i}{ 2 (k+ g_j )} = 1  \eea
for any $i=1,\ldots,s$.

\begin{cor}\label{Actsscalar}Assume $\k$ is simple, so that $\k=\k_1$. 
Then there is $k\in\C$ such that $\widetilde V(k,\k)$ is conformally embedded in $V_k(\g)$ if and only if $C_\k$ acts scalarly on $\p$.
\end{cor}
\begin{proof}If $\widetilde V(k,\k)$ is conformally embedded in $V_k(\g)$ then, 
by \eqref{numcheck2}, $\gamma_i^1=2(k+g_1)$ is independent of  $i$.
If $C_\k$ acts scalarly on $\p$, then, solving \eqref{numcheck2} for $k$, one finds a level where, by AP-criterion, conformal embedding occurs.
\end{proof}

\subsection{Dynkin index} Recall the notion of Dynkin index of a representation. Let $\g$ be  a simple Lie algebra and  $V$ a finite dimensional $\g$--module. Let  $tr_V$ be the trace form of $V$. It  defines a nondegenerate bilinear symmetric invariant form on $\g$, hence it is a multiple of the Killing form $\kappa$. The Dynkin index $Ind_\g(V)$ of $V$  is the multiplicative factor between the two forms; more precisely 
$$tr_V x^2= Ind_\g(V)\, \kappa(x,x),\quad x\in\g.$$
A  theorem of Dynkin states that if $V$ is an irreducible $\g$-module of highest weight $\mu$ then 
\begin{equation}\label{ind}Ind_\g(V)=\frac{\dim V}{\dim \g}\frac{(\mu,\mu+2\rho)}{(\theta,\theta+2\rho)},\end{equation}
where $(\cdot,\cdot)$ is any  nondegenerate bilinear symmetric invariant form on $\g$ and $\theta$ is the highest root of $\g$.  Let  $C^{(\cdot,\cdot)}_\g=\sum_{i=1}^{\dim\g}X_i^2$ be  the Casimir element corresponding to the form $(\cdot,\cdot)$. Let $\l,\l'$ be the eigenvalues of $C^{(\cdot,\cdot)}_\g$ acting on $V,\g$, respectively. Then 
$
Ind_\g(V)=\frac{\dim V}{\dim \g}\frac{\l}{\l'}.
$
In particular, choosing $(\cdot,\cdot)= tr_V$, we find that $\l=\frac{\dim \g}{\dim V}$.  Hence with obvious notation
\begin{equation}\label{ciserve}
Ind_\g(V)=\frac{1}{(\l')^{tr_V}}.
\end{equation}
Let $V$ be an irreducible $\g$--module; denote by $s(V)$ either $sl(V)$, or $so(V)$ if $V$ admits a nondegenerate bilinear symmetric invariant form.
\begin{lemma}\label{lemmat}Assume that  $C^{tr_V}_\g$ acts with at most  two eigenvalues $\l^{tr_V}_1,\l^{tr_V}_2$ on $s(V)$ according to the eigenspace decomposition $s(V)=\g\oplus\mathfrak p$. 
Assume also that 
\begin{equation}\label{rell} \l^{tr_V}_1>\l^{tr_V}_2+1	\end{equation}
Then $Ind_\g(V)<1$.
\end{lemma}
\begin{proof} By \eqref{ciserve}
$$Ind_\g(V)=\frac{1}{\l^{tr_V}_1}<\frac{\l^{tr_V}_1-\l^{tr_V}_2}{\l^{tr_V}_1}=1-\frac{\l^{tr_V}_2}{\l^{tr_V}_1}<1.$$
\end{proof}
\subsection{Classification of maximal reductive subalgebras  of a simple Lie algebra}\label{classDyn}
Recall that a Lie subalgebra $\k$ is said to be reductive in a Lie algebra $\g$ if the adjoint action of $\k$ on $\g$ is completely reducible. 
\par
\begin{defi}\label{mr}Let  $\g$ be a simple Lie algebra. We call a subalgebra $\k$ of $\g$ {\sl maximally reductive} if it  is maximal among subalgebras reductive in $\g$. 
\end{defi}
\noindent Note that a maximally reductive algebra need not  be a maximal subalgebra.\par
The next lemma shows that the classification  of maximally reductive  subalgebras can be reduced to Dynkin classification of  maximal semisimple subalgebras (i.e., semisimple subalgebras which are maximal among all subalgebras) and the classification of maximally reductive equal rank subalgebras.
\begin{lemma}\label{l1} (1). 
Suppose that $\k$ is semisimple and maximally reductive in $\g$. Then $\k$ is a maximal subalgebra of $\g$.
\par\noindent
(2). Suppose that $\k$ is maximally reductive in $\g$ and that it is not semisimple. Then $\k$ is an equal rank subalgebra.
\end{lemma}
\begin{proof} (1).
If $\k$ is maximally reductive but it is not maximal, then it is contained in a non-semisimple maximal subalgebra $\mathfrak{r}$ of $\g$. By a theorem of Morozov (see e.g. \cite{Mor}), $\mathfrak{r}$ is a parabolic subalgebra. By \cite[I, \S 6.8, Corollaire 2]{bour}, $\k$ is  a Levi component (=maximal semisimple subalgebra) $\mathfrak{m}$ of $\mathfrak{r}$. Write $\mathfrak{r}=\mathfrak{l}\oplus 	\mathfrak{n}$ with $\mathfrak{l}$ reductive in $\g$ and $\mathfrak{n}$ the nilradical of $\mathfrak{r}$. Then $\mathfrak{l}=[\mathfrak{l},\mathfrak{l}]\oplus \mathfrak{z}$, with $\mathfrak{z}$ the center of $\mathfrak{l}$.  Since $[\mathfrak{l},\mathfrak{l}]$ is a Levi component for $\mathfrak{r}$, by the Levi-Malcev theorem, $[\mathfrak{l},\mathfrak{l}]$ and $\k$ are conjugated by an inner automorphism $e^{ad(x)}, x\in\mathfrak{r}$, hence $\k\subset \k\oplus e^{ad(x)}(\mathfrak{z})$ and $\k\oplus e^{ad(x)}(\mathfrak{z})$ is reductive in $\g$. A contradiction.
\par
(2). 
Let $\mathfrak{z}\ne \{0\}$ be the center of $\k$. Then $\mathfrak{z}$ acts semisimply on $\g$, hence it is contained in a Cartan subalgebra $\h$ of $\g$. It follows that the centralizer of $\mathfrak{z}$ in $\g$ is $\mathfrak{l}+\h$,  $\mathfrak{l}$ being the Levi component of a parabolic subalgebra of $\g$. Since $\k$ is maximal reductive and $\k\subset \mathfrak{l}+\h$, we have $\k=\mathfrak{l}+\h$, hence it is an equal rank subalgebra.
\end{proof}

Combining Lemma \ref{l1}  with Dynkin's classification of maximal semisimple subalgebras of  a simple Lie algebra $\g$ (see  \cite{Dynk1}, \cite{Dynk}), we obtain the following description of   maximally reductive subalgebras.
\par Let $\k$ be a maximally reductive subalgebra of a simple Lie algebra $\g$. Then,
either $\k$ is equal rank subalgebra or, up to an inner automorphism of $\g$,  it falls in one of the following classes: 
\begin{enumerate}
\item if   $\k$ is simple,  then either $\k=so(2n-1)\subset\g=so(2n)$ or $\k\subset \g=so(V), sp(V), sl(V)$ with $V$  an irreducible $\k$--module except for the  cases  listed in 
\cite[Table 1]{Dynk1};
\item if $\g$ is of classical type and $\k$ is non-simple then $\k$  is  one of the subalgebras $\k$  in Table 2, in Section \ref{analysis};
\item if $\g$ is of exceptional type,  then $\k$ is  one of the algebras in  \cite[Theorem 14.1]{Dynk} (see also \cite{Min}).
\end{enumerate}
  \section{Conformal embeddings of $\widetilde V(k,\k)$ in  $V_k(\g)$ with $\k$ simple and $\g$ of classical type.}
 
\subsection{Conformal embeddings of $\widetilde V(k,\k)$ in  $V_k(so(V))$ with $\k$ simple.}

In this section we discuss the conformal embeddings of a simple Lie algebra $\k$ in $so(V)$. More specifically we consider an irreducible finite dimensional representation $V$ of $\k$ admitting a $\k$-invariant nondegenerate symmetric form $\langle\cdot,\cdot\rangle $.
 
 From now on we will denote $so(V,\langle\cdot,\cdot\rangle)$ simply as $so(V)$.

The conformal embeddings in $V_k(so(V))$ with $k\in \ganz_+$ and $V$ a finite dimensional representation are the subject of the symmetric space theorem. We recall this theorem
in a vertex algebra formulation.

\begin{thma}[{\bf Symmetric space theorem, \cite{GNO}}]{\it
Assume that a compact group $U$ with complexified Lie algebra $\mathfrak u$ acts faithfully on a finite dimensional complex space $V$ admitting a $U$--invariant symmetric nondegenerate form.

Then there is a conformal  embedding of $\widetilde V(k,\mathfrak u)$ in $V_k(so(V))$ with $k\in\ganz_+$ if and only if $k=1$ and 
there is a Lie algebra structure on $\mathfrak r=\mathfrak u\oplus V$ making $\mathfrak r$ semisimple, $(\mathfrak r,\mathfrak u)$ is a symmetric pair, and a nondegenerate invariant form on $\mathfrak r$ is given by the direct sum of a invariant form on $\mathfrak u$ with the chosen $U$--invariant form on $V$.}
\end{thma}

Note that the symmetric space theorem applies to a more general setting than ours as $\mathfrak u$ need not to be simple and $V$ is not necessarily irreducible. Following \cite{Kost}, we let  $\nu:\k\to so(V)$ be the representation map. We identify $\k$ and $\nu(\k)$. Let also $\tau:\bigwedge^2 V \to so(V)$ be  the $\k$--equivariant isomorphism such that $\tau(u)(v)=-2i(v)(u)$, where $i$ is the contraction map, extended to $\bigwedge V$ as an odd derivation. More explicitly
$
\tau^{-1}(X)=\frac{1}{4}\sum_i X(v_i)\wedge v_i,
$
where $\{v_i\}$ is an orthonormal basis of $V$.

Let $(\cdot,\cdot)$ be the normalized invariant form on $so(V)$. Recall that $(X,Y)=\half tr_V(XY)$.
As invariant symmetric form on $\k$ we choose  $(\cdot,\cdot)_{|\k\times \k}$. 
With notation as in Subsection \ref{APC}, let $g_1=\half\lambda_1^{(\cdot,\cdot)_{|\k\times \k}}=\half\l_1^{\half tr_V}=\l_1^{tr_V}$ be the eigenvalue for the action of $C_\k$ on $\k$. 
Let $\mathfrak p$ be the orthogonal  complement of $\k$ in $so(V)$. Let $p_\k$, $p_{\mathfrak p}$ be the orthogonal projections of $\bigwedge^2 V$ onto $\k$ and $\mathfrak p$ respectively.

We would like to classify all irreducible representations $V$ of $\k$ such that there is $k\in\C$ such that  $\widetilde V(k,\k)$ embeds  conformally in $V_k(so(V))$. According to Corollary \ref{Actsscalar}, $C_\k$ must act scalarly on $\mathfrak p$. 
Let $\gamma=2\l_2^{tr_V}$ be the eigenvalue for the action of $C_\k$ on $\mathfrak p$. 

To study conformal embeddings at non-integrable levels, we reformulate the symmetric space theorem using a criterion due to Kostant. 
Recall  the following notation from \cite{Kost}: let $\nu_*:\k\to\bigwedge^2 V$ the unique Lie algebra homomorphism such that $\tau\circ \nu_*=\nu$. Let $Cl(V)$ denote the Clifford algebra of $(V,\langle\cdot,\cdot\rangle)$. Extend $\nu_*$ to a Lie algebra homomorphism $\nu_*: \k\to Cl(V)$, hence to a homomorphism of associative algebras  $\nu_*: U(\k)\to Cl(V)$.
  Consider  $\nu_*(C_\k)=\sum_i\nu_*(X_i)^2$, where $\{X_i\}$ is an orthogonal basis of $\k$. Also recall that a pair $(\k,\nu)$ consisting of a 
Lie algebra $\k$ with a bilinear symmetric invariant form $(\cdot,\cdot)_\k$ and of a representation $\nu:\k\to so(V)$  is said to be 
 {\it of Lie type} if there is a Lie algebra structure on $\mathfrak r=\k\oplus V$, extending that of $\k$, such that 1) $[x,y]=\nu(x)y,\,x\in\k,\,y\in V$ and 2) the bilinear  form 
 $B_{\mathfrak r}=(\cdot,\cdot)_\k\oplus \langle \cdot, \cdot \rangle$ is $ad_{\mathfrak r}$--invariant. In \cite[Theorem 1.50]{Kost}  Kostant proved that $(\k,\nu)$ is a  pair of Lie type if and only if 
 there exists $v\in (\bigwedge^3 V)^\k$ such that 
 \begin{equation}\label{nu}\nu_*(C_\k)+v^2\in\C.\end{equation} Moreover he proved in \cite[Theorem 1.59]{Kost} that $v$ can be taken to be $0$ is and only if $(\k\oplus V,\k)$ is a symmetric pair.\par
 We now come to the reformulation of the Symmetric Space Theorem.
 \begin{prop}\label{strong}
Let $\k$ be a simple Lie algebra and $V$ an irreducible finite-dimensional representation of $\k$ admitting a nondegenerate symmetric $\k$--invariant form. 

Then there is $k\in\ganz_+$ such that $\widetilde V(k,\k)$ is conformally embedded in $V_k(so(V))$ if and only if 
\begin{equation}\label{Kosteq}
\sum_i \tau^{-1}(X_i)\wedge \tau^{-1}(X_i)=0,
\end{equation}
where $\{X_i\}$ is an orthogonal basis of $\k$.
 \end{prop}
 
 \begin{proof}
In \cite[Proposition 1.37]{Kost}
 it is shown that  $\nu_*(C_\k)$ might have nonzero components only in degrees $0,4$ w.r.t. the standard grading of $\bigwedge V \cong Cl(V)$.
Recall that, if $y\in V$ and $w\in Cl(V)$, then $y\cdot w=y\wedge w+i(y)w$, hence 
 $\nu_*(C_\k)=\sum_i \tau^{-1}(X_i)\wedge \tau^{-1}(X_i)+a$ with  $a\in\wedge^0V=\C$. Thus, 
 if \eqref{Kosteq} holds,
then,
$ \nu^*(C_\k)$ is a constant in $Cl(V)$.   Thus equation \eqref{nu} holds with $v=0$ and by Kostant's theorem quoted above (\cite[Theorem 1.59]{Kost}) 
we can conclude that 
$(\k\oplus V,\k)$ is a symmetric pair.  
Clearly we can assume that $V$ is not the trivial one-dimensional $\k$--module. Since $V$ is irreducible, we see that $V^\k=\{0\}$. Thus the hypothesis of \cite[Theorem 1.61]{Kost} are satisfied 
and $\mathfrak k\oplus V$ is semisimple.
By the symmetric space theorem, we know that $\widetilde V(1,\k)$ embeds conformally in $V_1(so(V))$. 

Conversely, if $\widetilde V(k,\k)$ is conformally embedded in $V_k(so(V))$ for some $k\in\ganz_+$, then, by the symmetric space theorem, there is a Lie algebra structure on $\mathfrak r=\k\oplus V$ making $\mathfrak r$ semisimple, $(\mathfrak r,\k)$ is a symmetric pair, and a nondegenerate invariant form on $\mathfrak r$ is given by the direct sum $(\cdot,\cdot)_{|\k\times\k}\oplus\langle\cdot,\cdot\rangle$. It follows from \cite[Theorem 1.59]{Kost} that \eqref{Kosteq} holds.
 \end{proof}
 

%
\begin{rem}

Let us provide an interpretation of Proposition  \ref{strong} and Kostant's results  in the context of the representation theory of affine vertex algebras. 

Recall that spin modules for the Lie algebra  of type $D_n$  are  irreducible   highest weight modules with highest weights  $\omega_n$ and $\omega_{n-1}$ (resp. $\omega_{n}$ if algebra is of type $B_n$).
Assume that $\widetilde V(k,\k)$ is conformally embedded in $V_k(so(V))$ and that the  spin $so(V)$--modules are modules for Zhu's algebra $A(V_k(so(V)) )$ for $V_k(so(V))$. Then one can show that there is a non-trivial homomorphism from $A(V_k(so(V)) )$ to
the Clifford algebra $Cl(V)$. By using the Kostant criterion,  we see that then (\ref{Kosteq}) holds. Applying  Proposition \ref{strong}, we get  $k\in\ganz_+$. 

So the only non-integrable candidates for realization of conformal embeddings  are vertex algebras $V_k(so(V))$ which do not admit  embeddings of Zhu's algebra to the Clifford algebra, and therefore do not admit spin modules. Examples of  such vertex algebras are $V_{-2} (B_3)$ and   $V_{-2} (D_4)$  which  provide  models for  realizations of non-integrable  conformal embeddings  (cf. \cite{A},  \cite{P-Glasnik}).
 \end{rem}

Conformal embeddings at non-integrable levels do occur. Consider the following cases (cf.  \cite[Section 5]{A}):
 \begin{itemize}
 
\item[$(B3)_{\omega_3}$]: $\k= B_3$, $ V= L_{B_3} (\omega_3)$  be the irreducible $8$--dimensional $B_3$--module. Then we have conformal embedding of $V_{-2} (B_3)$ into $V_{-2} (D_4) = V_{-2} (so(V))$.
\item[$(G_2)_{\omega_1}$]: $\k= G_2$, $V= L_{G_2} (\omega_1)$  be the irreducible $7$--dimensional $G_2$--module. Then we have conformal embedding of $V_{-2} (G_2)$ into $V_{-2} (B_3) = V_{-2} (so(V))$.

 \end{itemize} 


\begin{lemma}\label{critKost}
Assume that $\widetilde V(k,\k)$ embeds conformally in $V_k(so(V))$ with $k\not\in \ganz_+$. 
Then $k=-2$ and 
$$
\gamma=2\frac{(\dim V-4)\dim\k}{\dim so(V)},\quad g_1=\half(\gamma+4).
$$
\end{lemma}
\begin{proof}
Denote by $\lambda$  the eigenvalue of the action of $C_\k$ on $V$. Then 
\begin{align*}
\sum_i \tau^{-1}&(X_i)\wedge \tau^{-1}(X_i)=\frac{1}{16}\sum_{i,j,r} X_i(v_j)\wedge v_j\wedge X_i(v_r)\wedge v_r\\
&=-\frac{1}{16}\sum_{i,j,r} X_i(v_j)\wedge X_i(v_r)\wedge v_j\ \wedge v_r=-\frac{1}{32}\sum_{i,j,r} X^2_i(v_j\wedge v_r)\wedge v_j\ \wedge v_r\\&+\frac{1}{32}\sum_{i,j,r} X^2_i(v_j)\wedge v_r\wedge v_j\ \wedge v_r+\frac{1}{32}\sum_{i,j,r} v_j\wedge X^2_i( v_r)\wedge v_j\ \wedge v_r\\&=-\frac{1}{32}\sum_{i,j,r} X^2_i(v_j\wedge v_r)\wedge v_j\ \wedge v_r+\frac{\lambda}{16}\sum_{j,r} v_j\wedge v_r\wedge v_j\ \wedge v_r\\
&=-\frac{1}{32}\sum_{i,j,r} X^2_i(v_j\wedge v_r)\wedge v_j\ \wedge v_r.
\end{align*}

Write now
\begin{align*}
\sum_{i,j,r} X^2_i(v_j\wedge v_r)\wedge v_j\ \wedge v_r&=\sum_{i,j,r} X^2_i(p_\k(v_j\wedge v_r))\wedge v_j\ \wedge v_r+\sum_{i,j,r} X^2_i(p_{\mathfrak p}(v_j\wedge v_r))\wedge v_j\ \wedge v_r\\
&=2g_1\sum_{j,r} p_\k(v_j\wedge v_r)\wedge v_j\ \wedge v_r+\gamma\sum_{j,r} p_{\mathfrak p}(v_j\wedge v_r)\wedge v_j\ \wedge v_r\\
&=(2g_1-\gamma)\sum_{j,r} p_\k(v_j\wedge v_r)\wedge v_j\wedge v_r+\gamma\sum_{j,r} v_j\wedge v_r\wedge v_j\wedge v_r\\
&=(2g_1-\gamma)\sum_{j,r} p_\k(v_j\wedge v_r)\wedge v_j\wedge v_r.
\end{align*}

We now compute $p_\k(v_j\wedge v_r)$ explicitly. We extend $\langle\cdot,\cdot\rangle$ to $\bigwedge^2V$ by determinants. Set $u_i=\tau^{-1}(X_i)$ and note that
\begin{align*}\langle u_i,u_j\rangle&=\frac{1}{16}\sum_{r,s}det\begin{pmatrix}\langle X_i(v_r),X_j(v_s)\rangle& \langle X_i(v_r),v_s\rangle\\\langle v_r,X_j(v_s)\rangle& \langle v_r,v_s\rangle\end{pmatrix}\\
&=\frac{1}{16}\sum_{r}(\langle X_i(v_r),X_j(v_r)\rangle-\langle v_r,X_j(X_i(v_r))\rangle)\\&=-\frac{1}{8}\sum_{r}\langle X_j(X_i(v_r)),v_r\rangle=-\frac{1}{8}tr(X_jX_i).
\end{align*}
 
 Recall that $(X,Y)=\half tr_V(XY)$, hence $tr_V(X_jX_i)=2(X_j,X_i)$, so that $\langle u_i,u_j\rangle=-\frac{1}{4}\delta_{ij}$; therefore
\begin{align*}
 p_\k(v_j \wedge v_r)=&-4\sum_t\langle v_j\wedge v_r,u_t\rangle u_t=- \sum_{t,k}\langle v_j\wedge v_r,X_t(v_k)\wedge v_k\rangle u_t\\
 &=- \sum_{t,k}det\begin{pmatrix}\langle v_j,X_t(v_k)\rangle& \langle v_j,v_k\rangle\\\langle v_r,X_t(v_k)\rangle& \langle v_r,v_k\rangle\end{pmatrix} u_t=- 2\sum_{t}\langle v_j,X_t(v_r)\rangle u_t.
\end{align*}
 
 Substituting we find that
 \begin{align*}
\sum_i \tau^{-1}(X_i)\wedge \tau^{-1}(X_i)&=\frac{2g_1-\gamma}{16}\sum_{j,r,t} \langle v_j,X_t(v_r)\rangle u_t\wedge v_j\wedge v_r\\
&=\frac{2g_1-\gamma}{64}\sum_{j,r,t,k} \langle v_j,X_t(v_r)\rangle X_t(v_k)\wedge v_k\wedge v_j\wedge v_r\\
&=\frac{2g_1-\gamma}{64}\sum_{r,t,k}  X_t(v_k)\wedge v_k\wedge X_t(v_r)\wedge v_r\\&=\frac{2g_1-\gamma}{4}\sum_t \tau^{-1}(X_t)\wedge \tau^{-1}(X_t).
\end{align*}

By Proposition \ref{strong}, since  $k\not\in\ganz_+$, we have $\sum_t \tau^{-1}(X_t)\wedge \tau^{-1}(X_t)\ne0$, thus 
\begin{equation}
2g_1-\gamma={4}
\end{equation}

Recall that, by AP-criterion, $\frac{\gamma}{2(k+g_1)}=1$ so $k=\frac{\gamma-2g_1}{2}=-2$.

Finally, as central charges must be equal, we obtain
$$
\frac{\dim so(V)}{\dim V-4}=\frac{\dim \k}{g_1-2},
$$
so
\begin{equation}\label{lambdas}
g_1=\frac{(\dim V-4)\dim\k}{\dim so(V)}+2,\quad\gamma=2\frac{(\dim V-4)\dim\k}{\dim so(V)}.
\end{equation}
\end{proof}
 \begin{proposition}\label{ce} The conformal embeddings $(B3)_{\omega_3}, (G2)_{\omega_1}$ are the unique conformal embeddings of a simple Lie algebra in $V_{-2}(so(V))$ with $V$ irreducible.
 \end{proposition}
\begin{proof}  We will proceed to inspecting all cases; it will eventually turn out that cases $(B3)_{\omega_3},$ $(G2)_{\omega_1}$ are the only occurring. 
Let $\mu$ be a highest weight occurring in $\mathfrak p$.  First observe  $\mu$ belongs to the root lattice.
Indeed, recall that $V$ is irreducible and let $\l$ be its highest weight for some choice of a positive set of roots. Then, since $\bigwedge^2 V$ is a submodule of $V\otimes V$, the weights occurring in  $\bigwedge^2 V$ are of type $2\lambda+\eta$ with $\eta$ in the root lattice. Since   $\theta$ appears in the decomposition of $\bigwedge^2 V$ as a $\k$--module, we have that there is $\xi$ in the root lattice such that $2\lambda+\xi=\theta$, thus $2\lambda$ is in the root lattice. Thus any weight of $\bigwedge^2 V$ is in the root lattice.\par
Now, by \eqref{lambdas}, we have
$$\l_1^{tr_V}-\l_2^{tr_V}=\half(2g_1-\gamma)=2.$$
In turn, by Lemma \ref{ciserve} we have $Ind_\k V_\mu<1$ for any irreducible component $V_\mu$ of $\p$. Inspecting the list of irreducible representations with Dynkin index less than one given in \cite[Table 1]{AEV} and selecting the weights in the root lattice we have that  $\mu$ is necessarily a fundamental weight.  More precisely, using   Bourbaki's notation for Dynkin diagrams, we have : \par\vskip10pt
\centerline{
\begin{tabular}{ c | c |c |c |c }$\k$ & $B_n$ & $C_n$ & $F_4$ & $G_2$\\\hline
$\mu$ &$\omega_1$ & $\omega_2$ & $\omega_4$ & $\omega_1$ 
\end{tabular}
}
\vskip5pt
\centerline{\small Table 1}
\vskip10pt
If follows that every irreducible component of $\p$ must be of the form $L(\mu)$ with $\mu$ as in Table 1. Hence, as $\k$--modules
\begin{equation}\label{decso} so(V)=\k\oplus p\, L(\mu),\quad p\geq 1.\end{equation}
Let $(\cdot, \cdot)_n$ be  the normalized form of $\k$; setting $v=\dim V$, one has
\begin{equation}\label{l2norm}
\l_2^{(\cdot, \cdot)_n}=\frac{2 \dim\k\, h^\vee(v -4)}{\dim\k(v-4)+ v(v-1)}.
\end{equation}
Let us discuss type $G_2$. Taking dimensions in \eqref{decso} one has $v(v-1)/2=14+7 p$. Solving for $v$, the only positive solution is $v=(1+\sqrt{113+56 p})/2$.  Now we find, using \eqref{l2norm},  the values of $p$ such that 
$\l_2^{(\cdot, \cdot)_n}=(\omega_1,\omega_1+2\rho)=4$. One gets $p=1$ and $p=2$. In the former case one obtains $v=7$, which corresponds to $(G2)_{\omega_1}$; in the latter case one gets $v=8$, which is excluded since no irreducible representation of $G_2$ has dimension $8$. Type $F_4$ is  treated similarly: one gets a quadratic equation in $p$ which has no integral solution.\par
In type $B_n$ it is convenient to use a slightly different strategy. One checks, by direct computation, that $\l_1^{(\cdot, \cdot)_n}=4n-2, \l_2^{(\cdot, \cdot)_n}=2n, \gamma=\frac{4n}{n-1}$. Using the rightmost formula in \eqref{lambdas}, 
one obtains that $v=\frac{2(3n+p-2)}{n-1}=4+\frac{2(n+p)}{n-1}$, so that $p=q(n-1)-1$ and $v=2(3+q)$. Now equation $v(v-1)/2-\dim so(V)=0$ reads $n (1 + p) - 2 n^2 (1 + p) + 2 (8 + 6 p + p^2)=0$ or
$n-2n^2+2(5+p+\frac{3}{1+p})=0$. This implies $p\in\{1,2,5\}$, and one gets an integer value for $n$ only for $p=1$. More precisely, in that case $n=3$ and $v=8$, so we are in case $(B3)_{\omega_3}$.
Type $C_n$ is easier and treated similarly.
\end{proof}
  \subsection{Conformal embeddings of $\widetilde V(k,\k)$ in  $V_k(sp(V))$ with $\k$ simple.}
  
In this section we discuss the conformal embeddings of a simple Lie algebra $\k$ in $sp(V)$. More specifically we consider an irreducible finite dimensional representation $V$ of $\k$ admitting a $\k$--invariant nondegenerate symplectic form $\langle\cdot,\cdot\rangle $.
 
 From now on we will denote $sp(V,\langle\cdot,\cdot\rangle)$ simply as $sp(V)$.
We let  $\nu:\k\to sp(V)$ be the representation map. Let also $\tau:S^2 (V) \to sp(V)$ be  the linear isomorphism such that $\tau(u)(v)=i(v)(u)$, where $i$ is the contraction map, extended to $S^2( V)$ as an even derivation. More explicitly
$
\tau^{-1}(X)=\frac{1}{2}\sum_i X(v_i) v^i,
$
where $\{v_i\}$ is a basis of $V$ and $\{v^i\}$ is the corresponding dual basis (i.e. $\langle v_i,v^j\rangle=\delta_{ij}$).

Recall that a Lie superalgebra $\mathfrak r=\mathfrak r_0\oplus \mathfrak r_1$  is said to be of Riemannian type if it admits a nondegenerate supersymmetric even invariant form.

 \begin{prop}\cite{Kostsuper}\label{Ksuper}
Let $\k$ be a Lie algebra admitting a nondegenerate $ad$--invariant symmetric bilinear form $(\cdot,\cdot)$ and  $V$ an irreducible finite-dimensional representation of $\k$ admitting a nondegenerate symplectic $\k$--invariant form $\langle \cdot,\cdot\rangle$.

Then the space $\k\oplus V$ admits a Lie superalgebra structure such that the pair $(\k\oplus V,(\cdot,\cdot)\oplus\langle\cdot,\cdot\rangle)$ is of Riemannian type if and only if
\begin{equation}
\sum_i \tau^{-1}(X_i)^2=0,
\end{equation}
where $\{X_i\}$ is an orthogonal basis of $\k$.
 \end{prop}

Let us return to our situation: we are assuming that $\k$ is a simple  Lie algebra and $V$ is an irreducible finite-dimensional representation of $\k$ admitting a nondegenerate invariant symplectic form. Let $(\cdot,\cdot)$ be the normalized invariant bilinear form on $sp(V)$. Recall that $(X,Y)=tr_V(XY)$.
For an  invariant symmetric bilinear form on $\k$ we choose  $(\cdot,\cdot)_{|\k\times \k}$.  Let $\mathfrak p$ be the orthogonal  complement of $\nu(\k)$ in $sp(V)$. 
Then $g_1=\half\l_1^{tr_V}$. By Corollary \ref{APC},  $C_\k$ acts scalarly on $\mathfrak p$ and the eigenvalue is
 $\gamma=\l_2^{tr_V}$. 

\begin{lemma}\label{critKostsuper}
Assume that $\widetilde V(k,\k)$ embeds conformally in $V_k(sp(V))$. 
Then either $\k+V$ admits the structure of a Lie superalgebra of Riemannian type, or 
 $k=1$ and 
$$
\gamma=\frac{(\dim V+4)\dim\k}{\dim sp(V)},\quad g_1=\half\gamma-1.
$$
\end{lemma}\begin{proof}
 Let $p_\k$, $p_{\mathfrak p}$ be the  projections of $S^2 (V)$ onto $\k$ and $\mathfrak p$ respectively corresponding to the direct sum $S^2(\k)=\k\oplus\mathfrak p$.

Write explicitly
\begin{align*}
\sum_i (\tau^{-1}&(X_i))^2 =\frac{1}{4}\sum_{i,j,r} X_i(v_j)X_i(v_r)v^j v^r\\&=\frac{1}{8}\sum_{i,j,r} X^2_i(v_j v_r)v^jv^r-\frac{1}{8}\sum_{i,j,r} X^2_i(v_j)v_r v^jv^r-\frac{1}{8}\sum_{i,j,r} v_jX^2_i( v_r) v^j v^r\\
&=\frac{1}{8}\sum_{i,j,r} X^2_i(v_j v_r)v^j v^r-\frac{\lambda}{4}\sum_{j,r} v_j v_r v^j v^r.\end{align*}
Here $\lambda$ is the eigenvalue of the action of $C_\k$ on $V$. Noting that $\sum_{i} v_i  v^i=0$  we obtain $\sum_{j,r} v_j v_r v^j v^r=0$, hence
$
\sum_i (\tau^{-1}(X_i))^2=\frac{1}{8}\sum_{i,j,r} X^2_i(v_jv_r)v^j v^r.
$
Write now
\begin{align*}
\sum_{i,j,r} X^2_i(v_j v_r) v^j v^r&=\sum_{i,j,r} X^2_i(p_\k(v_j v_r)) v^j v^r+\sum_{i,j,r} X^2_i(p_{\mathfrak p}(v_j v_r)) v^j v^r\\
&=2g_1\sum_{j,r} p_\k(v_j v_r) v^j v^r+\gamma\sum_{j,r} p_{\mathfrak p}(v_j v_r) v^j v^r\\
&=(2g_1-\gamma)\sum_{j,r} p_\k(v_j v_r) v^j v^r+\gamma\sum_{j,r} v_j v_r v^j v^r\\
&=(2g_1-\gamma)\sum_{j,r} p_\k(v_j v_r) v^j v^r.
\end{align*}

We now compute $p_\k(v_j v_r)$ explicitly. We extend $\langle\cdot,\cdot\rangle$ to $S^2(V)$ by restricting to symmetric tensors the form $\langle \cdot,\cdot\rangle\otimes \langle \cdot,\cdot\rangle$ on $V\otimes V$. 

Set $u_i=\tau^{-1}(X_i)$ and note that
\begin{align*}\langle u_i,u_j\rangle&=\frac{1}{8}\sum_{r,s}(\langle X_i(v_r),X_j(v_s)\rangle\langle v^r,v^s\rangle+\langle X_i(v_r),v^s\rangle\langle v^r,X_j(v_s)\rangle)\\
&=\frac{1}{8}\sum_{r,s}(-\langle X_jX_i(v_r),v_s\rangle\langle v^r,v^s\rangle+\langle X_i(v_r),v^s\rangle\langle v^r,X_j(v_s)\rangle)\\
&=\frac{1}{8}\sum_{r}(\langle v^r,X_jX_i(v_r)\rangle+\langle v^r,X_jX_i(v_r)\rangle)=-\frac{1}{4}tr(X_jX_i).
\end{align*}
 
Recall that $tr(X_jX_i)=(X_j,X_i)$, hence $\langle u_i,u_j\rangle=-\delta_{ij}\frac{1}{4}$; therefore
\begin{align*}
 p_\k(v_j  v_r)=&- 4\sum_t\langle v_j v_r,u_t\rangle u_t=- 2\sum_{t,k}\langle v_j v_r,X_t(v_k) v^k\rangle u_t\\
 &=- \sum_{t,k}(\langle v_j,X_t(v_k)\rangle \langle v_r,v^k\rangle+\langle v_r,X_t(v_k)\rangle\langle v_j,v^k\rangle) u_t=- 2\sum_{t}\langle v_j,X_t(v_r)\rangle u_t.
\end{align*}
 
 Upon substituting, we find that
 \begin{align*}
\sum_i (\tau^{-1}(X_i))^2&=-\frac{2g_1-\gamma}{4}\sum_{j,r,t} \langle v_j,X_t(v_r)\rangle u_t v^j v^r\\
&=-\frac{2g_1-\gamma}{8}\sum_{j,r,t,k} \langle v_j,X_t(v_r)\rangle X_t(v_k) v^k v^j v^r\\
&=-\frac{2g_1-\gamma}{8}\sum_{r,t,k}  X_t(v_k)v^k X_t(v_r) v^r\\&=-\frac{2g_1-\gamma}{2}\sum_t (\tau^{-1}(X_t))^2.
\end{align*}

If $\sum_t (\tau^{-1}(X_t))^2=0$, then, by Proposition \ref{Ksuper}, $\k+V$ has the structure of a Lie superalgebra of Riemannian type.
Otherwise, 
we must have that 
\begin{equation}
2g_1-\gamma={-2}.
\end{equation} 

Recall that, by AP-criterion, $\frac{\gamma}{2(k+g_1)}=1$ so $k=\frac{\gamma-2g_1}{2}=1$.
Finally, as central charges must be equal, we obtain
$$
\frac{\dim sp(V)}{\dim V/2+2}=\frac{\dim \k}{g_1+1},
$$
so
\begin{equation}
g_1=\frac{(\dim V/2+2)\dim\k}{\dim sp(V)}-2,\quad\gamma=\frac{(\dim V+4)\dim\k}{\dim sp(V)}.
\end{equation}

 \end{proof}

 \begin{proposition}\label{cesp} Let $\k$ be a simple Lie algebra and $V$ an irreducible symplectic representation of $\k$. If  $\widetilde V(k,\k)$ is conformally embedded in $V_k(sp(V))$, then either $k=1$ or $\k=sp(V)$.
 \end{proposition}
\begin{proof} Assume $k\ne 1$. By Lemma \ref{critKostsuper}, $\k+V$ admits the structure of a Lie superalgebra of Riemannian type. We claim that $\k'=\k+V$ is a simple Lie superalgebra. 
Assume $\mathfrak i$ is a graded ideal of $\k'$; then $\mathfrak i=(\mathfrak i\cap\k)\oplus(\mathfrak i\cap V)$  and since 
 $\k$ is simple and $V$ is irreducible, either $\k=\mathfrak i$ or $\mathfrak i = V$. In the former case, $[\mathfrak i,V]\subset \mathfrak i=\k$ and  $[\mathfrak i,V]=[\k,V]\subset V$. Hence $[\k,V]=0$, so that  
 $V$ is the $1$--dimensional trivial representation, which is not symplectic. In the latter case then $[V,V]\subset \k$ and $[V,V]\subset\i=V$ hence $[V,V]=0$. By   invariance  of the bilinear form of $\k'$, we have that 
 $([\k,V],V)=(\k,[V,V])=0$; by non-degeneracy $[\k,V]=0$, hence again $V$ is the  trivial representation. By Kac classification of finite dimensional simple Lie superalgebras \cite{Kacsuper}, there are three simple Lie superalgebras $\mathfrak l$, for which $\mathfrak l_{\ov 0}$ is simple:
$B(0,n)$ and the strange ones, $P(n)$ and $Q(n)$.  But the strange ones have no even non-zero invariant bilinear form. 
We are left with $B(0,n)$, which corresponds to $\k'=sp(V)\oplus V$.
 \end{proof} 
 
 \begin{rem}
 In the case when $\k$ is semisimple but not simple, we can  have  conformal embeddings in $V_k(sp(V))$  when $k =-1/2$. One example is $\k = sl(2) \times so(m)$ and $V$ is $2m$--dimensional  irreducible $\k$--module isomorphic to the tensor product 
 $V_{sl(2)} (\omega_1) \otimes V_{so(m)} (\omega_1)$. Such conformal embeddings will be studied in  Sections \ref{analysis} and \ref{analysis-II}. 
 \end{rem}
 
   \subsection{Conformal embeddings of $\widetilde V(k,\k)$ in  $V_k(sl(V))$ with $\k$ simple.}
   
In this section we discuss the conformal embeddings of a simple Lie algebra $\k$ in $sl(V)$. More specifically we consider an irreducible finite dimensional representation $V$ of $\k$.
 
 Let  $\tau:V\otimes V^* \to gl(V)$ be  the linear isomorphism such that $\tau(v\otimes \l)(w)=\lambda(w)v$. More explicitly
$
\tau^{-1}(X)=\sum_i X(v_i) \otimes v^i,
$
where $\{v_i\}$ is a basis of $V$ and $\{v^i\}$ is the corresponding dual basis in $V^*$. Let also $\tau^*$ be the corresponding map form $V^*\otimes V$ to $gl(V^*)$.
Restricting $\tau^{-1}$ to $sl(V)$ we obtain a monomorphism of $\k$--modules $\tau^{-1} : sl(V)\to V\otimes V^*$.

Assume that $\widetilde V(k,\k)$ embeds conformally in $V_k(sl(V))$.
We choose $(\cdot,\cdot)$ to be the normalized invariant form on $sl(V)$. Recall that $(X,Y)=tr_V(XY)$.
As an invariant symmetric form on $\k$ we choose  $(\cdot,\cdot)_{|\k\times \k}$, so that $g_1=\half \lambda_1^{tr_V}$ and $\gamma=\lambda_2^{tr_V}$. 

\begin{lemma}\label{critKostsuperslV}
Assume that $\widetilde V(k,\k)$ embeds conformally in $V_k(sl(V))$. Then either $\k=sl(V)$ or  $k=\pm 1$. Moreover, 
\begin{equation}\label{lambdaslv}
g_1=\frac{\dim\k}{\dim V\mp 1} \mp 1,\quad\gamma=\frac{2\dim\k}{\dim V\mp 1},
\end{equation}
where we take the upper sign for $k=1$, the lower sign for $k=-1$.
\end{lemma}\begin{proof}
 Let $p_\k, p_{\mathfrak p}, p_{\C I_V}$ be the  projections of $V\otimes V^*$ onto $\k, \mathfrak p$ and $\C I_V$ respectively corresponding to the direct sum $V\otimes V^*=\k\oplus\mathfrak p\oplus \C \sum_i v_i\otimes v^i$.
Write explicitly
\begin{align*}
&\sum_i \tau^{-1}(X_i)\otimes  (\tau^*)^{-1}(X_i) =\sum_{i,j,r} X_i(v_j)\otimes v^j\otimes X_i(v^r)\otimes v_r\\
&=\sum_{i,j,r} \sigma_{23}(X_i(v_j)\otimes X_i( v^r)\otimes v^j\otimes v_r)\\
&=\half\sum_{i,j,r}\sigma_{23} (X^2_i(v_j\otimes v^r)\otimes v^j\otimes v_r)-\frac{1}{2}\sum_{i,j,r} X^2_i(v_j)\otimes v^j\otimes v^r \otimes v_r
\\
&-\frac{1}{2}\sum_{i,j,r} v_j\otimes v^j\otimes X^2_i( v^r)\otimes v_r\\
&=\frac{1}{2}\sum_{i,j,r} \sigma_{23}(X^2_i(v_j\otimes v^r)\otimes v^j \otimes v_r)-\lambda\sum_{j,r} v_j\otimes  v^j\otimes v^r \otimes v_r.
\end{align*}
Here $\lambda$ is the eigenvalue of the action of $C_\k$ on $V$. 
Write now
\begin{align*}
\sum_{i,j,r} &\sigma_{23}(X^2_i(v_j\otimes v^r)\otimes v^j \otimes v_r)=\sum_{i,j,r} \sigma_{23}(X^2_i(p_\k(v_j\otimes v^r)) \otimes v^j \otimes v_r\\&+\sum_{i,j,r}  \sigma_{23}X^2_i(p_{\mathfrak p}(v_j \otimes v^r))\otimes  v^j \otimes v_r+\sum_{i,j,r}  \sigma_{23}X^2_i(p_{\C\sum v_i\otimes v^i}(v_j \otimes v^r))\otimes  v^j \otimes v_r\\
&=2g_1\sum_{j,r} \sigma_{23}(p_\k(v_j \otimes v^r)\otimes  v^j \otimes v_r)+\gamma\sum_{j,r}\sigma_{23}( p_{\mathfrak p}(v_j \otimes v^r)\otimes  v^j\otimes v_r).
\end{align*}
In the last equality we used the fact that $C_\k$ acts trivially on $\sum v_i\otimes v^i$. Thus
\begin{align*}\sum_{i,j,r} &\sigma_{23}(X^2_i(v_j\otimes v^r)\otimes v^j \otimes v_r)
=(2g_1-\gamma)\sum_{j,r}\sigma_{23}( p_\k(v_j\otimes  v^r) \otimes v^j\otimes  v_r)\\&+\gamma\sum_{j,r}\sigma_{23}(p_{sl(V)}( v_j \otimes v^r)\otimes v^j \otimes  v_r)
\end{align*}
To compute $p_{sl(V)}( v_j \otimes v^r)$ we use the trace form on $V\otimes V^*$, so that, if  $\langle v\otimes \lambda,w\otimes \mu\rangle=\mu(v)\lambda(w)$, then 
$$
p_{sl(V)}(X)=X-\frac{\langle X,\sum v_i\otimes v^i\rangle}{\langle \sum v_i\otimes v^i,\sum v_i\otimes v^i\rangle} \sum_i v_i\otimes v^i.
$$
In particular
$$
p_{sl(V)}(v_j\otimes v^r)=v_j\otimes v^r-\frac{\delta_{jr}}{\dim V} \sum_i v_i\otimes v^i.
$$

We now compute $p_\k(v_j\otimes v^r)$ explicitly. 
Set $u_i=\tau^{-1}(X_i)$ and note that
\begin{align*}\langle u_i,u_j\rangle&=\sum_{r,s}v^s( X_i(v_r))v^r(X_j(v_s)) =\sum_{s}v^s( X_jX_i(v_s))=tr(X_jX_i).
\end{align*}
Recall that $tr(X_jX_i)=(X_j,X_i)$, hence $\langle u_i,u_j\rangle=\delta_{ij}$; therefore
\begin{equation}\label{pg}
 p_\k(v_j \otimes v^r)=\sum_t\langle v_j \otimes v^r,u_t\rangle u_t=\sum_{t,k}\langle v_j \otimes v^r,X_t(v_k) \otimes v^k\rangle u_t=\sum_{t}v^r(X_t(v_j)) u_t.\end{equation}
 
 Substituting we find that
 \begin{align*}
&\sum_i \tau^{-1}(X_i)\otimes (\tau^*)^{-1}(X_i)\\&=\half(2g_1-\gamma)\sum_{j,r}\sigma_{23}(p_\k(v_j\otimes v^r)\otimes v^j\otimes v_r)+\half\gamma\sum_{j,r}\sigma_{23}(p_{sl(V)}( v_j \otimes v^r)\otimes v^j \otimes  v_r)\\&-\lambda\sum_{j,r} v_j\otimes  v^j\otimes v^r \otimes v_r
\\&=\half(2g_1-\gamma)\sum_{j,r,t}\sigma_{23}(v^r(X_t(v_j)))u_t\otimes v^j\otimes v_r)
+\half\gamma\sum_{j,r} v_j \otimes v^j\otimes v^r \otimes  v_r\\&-\half\gamma\sum_{j,i}\frac{1}{\dim V}v_i\otimes v^j\otimes v^i\otimes v_j
-\lambda\sum_{j,r} v_j\otimes  v^j\otimes v^r \otimes v_r\\
&=-\half(2g_1-\gamma)\sum_{r,t,i}X_t(v_i)\otimes X_t(v^r)\otimes v^i\otimes v_r)
+\half\gamma\sum_{j,r} v_j \otimes v^j\otimes v^r \otimes  v_r\\&-\half\gamma\sum_{j,i}\frac{1}{\dim V}v_i\otimes v^j\otimes v^i\otimes v_j
-\lambda\sum_{j,r} v_j\otimes  v^j\otimes v^r \otimes v_r.
\end{align*}

 Now write 
\begin{align*}
 \sum_{t,r,i} &X_t(v_i)\otimes X_t(v^r)\otimes v^i \otimes v_r=\half \sum_{t,r,i} X_t^2(v_i\otimes v^r)\otimes v^i \otimes v_r -\lambda \sum_{r,i} v_i\otimes v^r\otimes v^i \otimes v_r\\
 &=\half(2g_1-\gamma)\sum_{r,i} p_\k(v_i\otimes v^r)\otimes v^i \otimes v_r +\half\gamma\sum_{r,i} p_{sl(V)}(v_i\otimes v^r)\otimes v^i \otimes v_r\\&-\lambda \sum_{r,i} v_i\otimes v^r\otimes v^i \otimes v_r\\
&=\half(2g_1-\gamma)\sum_{r,i,t} v^r(X_t(v_i)) u_t\otimes v^i \otimes v_r +\half\gamma\sum_{r,i} v_i\otimes v^r\otimes v^i \otimes v_r\\&-\half\gamma \sum_{r,i,s} \frac{\delta_{ri}}{\dim V}v_s\otimes v^s\otimes v^i \otimes v_r-\lambda \sum_{r,i} v_i\otimes v^r\otimes v^i \otimes v_r\\
&=-\half(2g_1-\gamma)\sum_{r,t,j}  X_t(v_j)\otimes v^j \otimes X_t(v^r)\otimes v_r +\half\gamma\sum_{r,i} v_i\otimes v^r\otimes v^i \otimes v_r\\&-\half\gamma \sum_{r,s} \frac{1}{\dim V}v_s\otimes v^s\otimes v^r \otimes v_r-\lambda \sum_{r,i} v_i\otimes v^r\otimes v^i \otimes v_r.
\end{align*}

Summing up we obtain
\begin{align}\label{formulona}&
\sum_i \tau^{-1}(X_i)\otimes  (\tau^*)^{-1}(X_i) =\frac{(2g_1-\gamma)^2}{4}\sum_i \tau^{-1}(X_i)\otimes  (\tau^*)^{-1}(X_i)\\\notag &(-\frac{(2g_1-\gamma)\gamma}{4}+\frac{(2g_1-\gamma)\lambda}{2}-\frac{\gamma}{2\dim V})\sum_{r,i} v_i\otimes v^r\otimes v^i \otimes v_r\\\notag &+(\frac{(2g_1-\gamma)\gamma}{4\dim V} +\half\gamma-\lambda)\sum_{r,s}v_s\otimes v^s\otimes v^r \otimes v_r.
\end{align}


Consider the natural identifications
$$\iota: V\otimes V^*\otimes V^*\otimes V\to gl(V)\oplus gl(V^*)\to gl(V)\otimes gl(V)^*\to End(gl(V))$$
Also recall that $gl(V)=\k\oplus\mathfrak p\oplus \C I_V$. We claim that
\begin{align}\label{prima}
&\iota(\sum_i \tau^{-1}(X_i)\otimes  (\tau^*)^{-1}(X_i))=-p_\k,\\\label{seconda}
&\iota(\sum_{r,i} v_i\otimes v^r\otimes v^i \otimes v_r)=I_{gl(V)},\\\label{terza}
&\iota(\sum_{r,s}v_s\otimes v^s\otimes v^r \otimes v_r)= (\dim V) p_{\C I_V}.
\end{align}
We check \eqref{prima}:
\begin{align*}
\iota(\sum_i \tau^{-1}(X_i)\otimes  (\tau^*)^{-1}(X_i))(v_r\otimes v^s)&=\sum_{i,h,k}(X_i(v^h)\otimes v_h)(v_r \otimes v^s) X_i(v_k)\otimes  v^k
 \\
 &=-\sum_{i,h,k}v^h(X_i(v_r)) v^s(v_h)X_i(v_k)\otimes  v^k
 \\
 &=-\sum_{i,k}v^s(X_i(v_r))X_i(v_k)\otimes  v^k,
\end{align*}
which is, by \eqref{pg}, the required relation. 
Relations \eqref{seconda}, \eqref{terza} are proven with a similar and easier  straightforwad computation.\par
Hence we conclude that the three tensors $\sum_i \tau^{-1}(X_i)\otimes  (\tau^*)^{-1}(X_i), \sum_{r,i} v_i\otimes v^r\otimes v^i \otimes v_r, \sum_{r,s}v_s\otimes v^s\otimes v^r \otimes v_r$ are linearly independent 
unless $\mathfrak p=0$. In the latter case $\k=sl(V)$; in the former we get, in particular, 
\begin{equation}\label{ll}
\lambda=\frac{(2g_1-\gamma)\gamma}{4\dim V} +\half\gamma.
\end{equation}
Substituting \eqref{ll} in the coefficient of $\sum_{r,i} v_i\otimes v^r\otimes v^i \otimes v_r$ in \eqref{formulona}  we obtain for this coefficient the expression
$\l_2((\l_1-\l_2)^2-4)/(8\dim V).$ Thus \eqref{formulona} becomes 
\begin{align}\label{formulona2}
\sum_i \tau^{-1}(X_i)\otimes  (\tau^*)^{-1}(X_i) &=\frac{(2g_1-\gamma)^2}{4}\sum_i \tau^{-1}(X_i)\otimes  (\tau^*)^{-1}(X_i)\\\notag &+\frac{\l_2((\l_1-\l_2)^2-4)}{8\dim V}\sum_{r,i} v_i\otimes v^r\otimes v^i \otimes v_r.
\end{align}
Hence we obtain 
\begin{equation}\label{finale}
2g_1-\gamma=\pm 2.
\end{equation}
By AP-criterion, $\frac{\gamma}{2(k+g_1)}= 1$ so $k=\frac{\gamma-2g_1}{2}=\pm 1$.

Finally, as central charges must be equal, we obtain (upper sign for $k=1$, lower sign for $k=-1$)
$$
\frac{\dim sl(V)}{\dim V\pm 1}=\frac{\dim \k}{g_1\pm1},
$$
which gives \eqref{lambdaslv}.
 \end{proof}
 Recall from  \cite{AP-Sigma} the following result.
  \begin{itemize}
 \item[$(C_n)_{\omega_1}$:] let $\k$ be of type $C_n$, let $V= L_{C_n} (\omega_1)$  be  the irreducible $2n$--dimensional $C_n$--module. Then we have a conformal embedding of $V_{-1} (C_n)$ into $V_{-1} (A_{2n-1}) = V_{-1} (sl(V))$.
 \end{itemize} 
 \begin{proposition}\label{cesl} Let $\k$ be a simple Lie algebra and $V$ an irreducible representation of $\k$. If  $\widetilde V(k,\k)$ is conformally embedded in $V_k(sl(V))$, then either $k=1$ or $\k=sl(V)$ or
we are in case $(C_n)_{\omega_1}$.
\end{proposition}
\begin{proof} We are reduced, by Lemma \ref{critKostsuperslV}, to deal with the case $k=-1$. Arguing as in Proposition \ref{ce}, we have that 
\begin{equation}\label{equaz}sl(V)=\k\oplus p L(\mu), \text{where }Êp\geq 1,\end{equation}
with $\mu$ as in Table 1. 
Let $(\cdot, \cdot)_n$ be  the normalized form of $\k$; setting $v=\dim V$, one has
\begin{equation}\label{l2norm2}
\l_2^{(\cdot, \cdot)_n}=\frac{\gamma}{g_1} h^\vee=\frac{2(\dim\k) h^\vee}{1+v+\dim\k}.
\end{equation}
Let us discuss type $G_2$. Since $\mu=\omega_1$, we know that $\l_2^{(\cdot, \cdot)_n}=4$. Hence, by  \eqref{l2norm2}, we have that  $v=13$.  But no irreducible representation of $G_2$ has that dimension.
Type $F_4$ is dealt with similarly. \par
For type $C_n$, we have $\mu=\omega_2$ and $\l_2^{(\cdot, \cdot)_n}=2n$, hence \eqref{l2norm2} gives   $v=2n$. This implies $V= L(\omega_1)$ and we get the conformal embedding 
$(C_n)_{\omega_1}$.\par
In type $B_n$, we have $\mu=\omega_1$ and $\l_2^{(\cdot, \cdot)_n}=2n$, hence \eqref{l2norm2} gives   $v=2n^2-n-2$. Hence, by dimensional considerations and the results from \cite{AEV}, 
either $V$ is the little adjoint representation (i.e., the irreducible module whose highest weight is the highest short root) or it is the spin representation when $n=3,4,5,6$. The former case is excluded since $2n^2-n-2\ne 2n+1$ for all $n$, The latter  case is excluded similarly except when $n=6$. In that case dimensions match, but $L(\omega_6)\otimes L(\omega_6)$ has a component of type $L(2\omega_6)$, which contradicts \eqref{equaz}  .
\end{proof}

  \section{Conformal embeddings of $\widetilde V(k,\k)$ in  $V_k(\g)$ with $\g$ classical and $\k$ semisimple non-simple.}
  \label{analysis}
  
  By Dynkin's theory (see Subsection \ref{classDyn}),  the only possibilities for a maximal non-equal rank  semisimple non-simple subalgebra $\k$ of $\g$ with $\g$ classical are given in the first two columns of Table 2. By equating central charges we find the values for the level $k$ when conformal embedding might occur. We list these values in the third column of Table 2.
  
\centerline{\begin{tabular}{c|c|c| c}
$\g$& $\k$ & $k$ & Ref\\
\hline
$sl(mn)$ & $sl(n)\times sl(m)$ & $1$ for all $n,m$;$-1$ for $n\ne m$ & $(slsl)$\\
\hline
$so(4mn)$ & $sp(2n)\times sp(2m)$ & $1,-\frac{2(mn-1)}{2mn-m-n}$& $(spsp)$\\
\hline
$so(mn)$ & $so(n)\times so(m)$ & $1,\frac{4-mn}{m+n+mn}$& $(soso)$\\
\hline
$sp(2mn)$ & $sp(2n)\times so(m)$ & $-\frac{1}{2}, \frac{mn+2}{2mn+2n-m}$& $(spso)$\\
\hline
$so(2(m+n+1))$ & $so(2n+1)\times so(2m+1)$ & $1,1-m-n$& $(BB)$\\
\end{tabular}}
\vskip5pt
\centerline{\small Table 2}
\vskip5pt

  \begin{theorem}  \label{thm1.1}
  (1) The embedding $sl(n)\times sl(m)\subset sl(nm)$ is conformal for $k=1$. If $n\ne m$, it is conformal also for $k=-1$. 
  (If $n=m$, then the level $k=-1$ is excluded, since it is critical for the factors $sl(n)$ of $sl(n)\times sl(n)$.)
  
 (2)  The embedding $sp(2n)\times sp(2m)\subset so(4nm)$ is conformal only for $k=1$. 
   (If $n=m$, then the level $k=-1-\frac{1}{m}$ from $(spsp)$ is critical for the factors $sp(2n)$ of $sp(2n)\times sp(2n)$.)

(3)  The embedding $so(n)\times so(m)\subset so(nm)$ is conformal only for $k=1$. 
  (If $n=m$, then the level $k=-1+\frac{2}{m}$ from $(soso)$ is critical for the factors $so(n)$ of $so(n)\times so(n)$.) 

(4)  The embedding $sp(2n)\times so(m)\subset sp(2nm)$ is conformal only for $k=-1/2$, except when $n=1$ in which case it is conformal for $k=-1/2,1$.
 (If $m=2n +2$, then the level $k=-1/2$ from $(spso)$ is critical for both factors  of $sp(2n)\times so(2n+2)$.) 
  
  (5)  The embedding $so(2n+1)\times so(2m+1)\subset so(2(n+m+1))$ is conformal  for $k=1$ and $k=1-n-m$, except when $n=m$ in which case it is conformal only for $k=1$.  (If $m=n$, then the level $k=1-m-n$  is critical for both factors $so(2n+1)$.)
  
  \end{theorem}
  \begin{rem} The knowledge of decompositions at level $1$ is as follows: they are known in a very explicit combinatorial way \cite{Ost} for $(slsl)$; a  formula for cases
 $(spsp)$, $(soso)$ involving the combinatorics  of Borel stable abelian subspaces  is given in \cite{CKMP}; a  general formula including also $(slsl)$ is given in \cite{KMPX}.
\end{rem}

  \begin{proof}[Proof of Theorem \ref{thm1.1}]
  We apply the AP-criterion. This requires to describe $\p$ as a $\k$--module, which can be done by classical invariant theory.
  
  Case $(slsl)$.  We realize  $sl(nm)=Hom((\C^n\otimes \C^m),(\C^n\otimes \C^m))/\C$ as a $sl(n)\times sl(m)$--module. 
  Since $\C^n\otimes \C^m=L_{sl(n)}(\omega_1)\otimes L_{sl(m)}(\omega_1)$ then 
  $$
  End((\C^n\otimes \C^m))=(L_{sl(n)}(\omega_1)\otimes L_{sl(n)}(\omega_1))\otimes( L_{sl(m)}(\omega_1)\otimes L_{sl(m)}(\omega_1)).
  $$
   Observe that
  $$
   L_{sl(n)}(\omega_1)\otimes L_{sl(n)}(\omega_1)=L_{sl(n)}(\theta)\oplus L_{sl(n)}(0),
  $$
hence
\begin{align*}
sl(nm) &= 
(L_{sl(n)}(\theta)\otimes L_{sl(m)}(0))\oplus(L_{sl(n)}(0)\otimes L_{sl(n)}(\theta))\oplus(L_{sl(n)}(\theta)\otimes L_{sl(n)}(\theta))\\
&= (sl(n)\times sl(m))\oplus (L_{sl(n)}(\theta)\otimes L_{sl(n)}(\theta)).
\end{align*}
and 
$$
\p=L_{sl(n)}(\theta)\otimes L_{sl(n)}(\theta).
$$
We need henceforth to check that
$$
\frac{(\theta,\theta+2\rho)_{sl(n)}}{2m(k+n/m)}+\frac{(\theta,\theta+2\rho)_{sl(m)}}{2n(k+m/n)}=1.
$$
Here $(\cdot,\cdot)_{sl(p)}$ is the trace form on $sl(p)$. Since
$(\theta,\theta+2\rho)_{sl(p)}=2p$ and we have
$$
\frac{n}{mk+n}+\frac{m}{nk+m}=1
$$
which has solutions $k=\pm1$ except when $n=m$, where the only solution is $k=1$.
  
 Case $(spsp)$. We realize  $so(4nm)=\bigwedge^2((\C^{2n}\otimes \C^{2m}))$ as a $sp(2n)\times sp(2m)$--module. 
  Since 
  \begin{equation}\label{l2}
  {\bigwedge}^2(\C^{2n}\otimes \C^{2m})=\left( {\bigwedge}^2(\C^{2n})\otimes S^2(\C^{2m})\right)\oplus \left({\bigwedge}^2(\C^{2m})\otimes S^2(\C^{2m})\right),
\end{equation}
${\bigwedge}^2(\C^{2r}) =L_{sp(2r)}(\omega_2)\oplus L_{sp(2r)}(0)$, and $S^2(\C^{2r})=L_{sp(2r)}(\theta)$, we obtain that

$$
so(4mn)=(L_{sp(2n)}(\theta)\otimes L_{sp(2m)}(\omega_2))\oplus (L_{sp(2n)}(\omega_2)\otimes L_{sp(2m)}(\theta))\oplus (sp(2n)\times sp(2m)).
$$  
and
$$\p=(L_{sp(2n)}(\theta)\otimes L_{sp(2m)}(\omega_2))\oplus(L_{sp(2n)}(\omega_2)\otimes L_{sp(2m)}(\theta))$$
We need henceforth to check that
$$
\frac{(\theta,\theta+2\rho)_{sp(n)}}{2m(k+(n+1)/m)}+\frac{(\omega_2,\omega_2+2\rho)_{sp(m)}}{2n(k+(m+1)/n)}=1.
$$
and 
$$
\frac{(\omega_2,\omega_2+2\rho)_{sp(n)}}{2m(k+(n+1)/m)}+\frac{(\theta,\theta+2\rho)_{sp(m)}}{2n(k+(m+1)/n)}=1.
$$
Here $(\cdot,\cdot)_{sp(r)}$ is the normalized invariant bilinear  form of   $sp(r)$. 

It is readily verified that the solutions for the first equation are 
$1$ and $-(m+1)/m$, and $1$ and $-(n+1)/n$ for the second equation. If $m=n$ the values different from $1$ coincide with value given in the table; on the other hand this value is critical. 

Case $(soso)$. We realize  $so(nm)=\bigwedge^2(\C^{n}\otimes \C^{m})$ as a $so(n)\times so(m)$--module. Using \eqref{l2}Ê(with $n,m$ in place of $2n,2m$), we have that 
$$so(mn)=(L_{so(n)}(\theta)\otimes L_{so(m)}(2\omega_1))\oplus ( L_{so(n)}(2\omega_1)\otimes L_{so(m)}(\theta))\oplus so(n)\times so(m),$$
so that 
$$\p=(L_{so(n)}(\theta)\otimes L_{so(m)}(2\omega_1))\oplus(L_{so(n)}(2\omega_1)\otimes L_{so(m)}(\theta)),$$ and calculations similar to the previous case lead to statement (3).

Case $(spso)$. We realize  $sp(2nm)=S^2((\C^{2n}\otimes \C^{m}))$ as a $so(n)\times so(m)$--module. We  have 
$$S^2(\C^{2n}\otimes \C^{m})=\left(S^2(\C^{2n})\otimes S^2(\C^{m})\right)\oplus \left({\bigwedge}^2(\C^{2n})\otimes {\bigwedge}^2(\C^{m})\right),$$
and 
$$sp(2nm)= (L_{sp(2n)}(\theta)\otimes L_{so(m)}(2\omega_1))\oplus (L_{sp(2n)}(\omega_2)\otimes L_{so(m)}(\theta))\oplus(sp(2n)\times so(m)),$$
so that
$$\p=(L_{sp(2n)}(\theta)\otimes L_{so(m)}(2\omega_1))\oplus (L_{sp(2n)}(\omega_2)\otimes L_{so(m)}(\theta)).$$
Again an easy calculation leads to (4).

Case $(BB)$. In this case $\k$  is the fixed point set of an involutive automorphism of $\g$. From this observation is easy to derive the decomposition
$$
\g=\k\oplus (L_{so(2n+1)}(\omega_1)\otimes  L_{so(2m+1)}(\omega_1)),
$$
hence we can verify that this  embedding is conformal by applying the AP-criterion. 
 \end{proof}
 
 \begin{rem}  The branching rules for the  conformal embedding in the cases $m=2$, $n \ge 3$ of Theorem \ref{thm1.1}(4)  were studied in   \cite{AP-JAA}. It was proved that then $V_{-1/2} (sp (4n) )$ is a semisimple $V_{-1} (sp(2n)) \otimes M(1)$--module, where $M(1) = V_1(so(2))$ is a Heisenberg vertex algebra of central charge $c=1$. So it is natural to conjecture that in many cases we shall get semi-simple decomposition. But as we shall see below, we have some examples where  we will not have semisimplicity. 
 \end{rem}
  
%

 \vskip 5mm 
It is natural to ask about semisimplicity of $V_{-1}(sl(m n) )$ as a $V^{-m} (sl(n) ) \otimes V^{-n} (sl(m) )$--module. Let $$\Phi :  V^{-m} (sl(n) ) \otimes V^{-n} (sl(m) ) \rightarrow V _{-1}(sl(m n) )$$  be the associated homomorphism of vertex algebras. Clearly, if $ m > n$, then $V^{-m} (sl(n) ) = V_{-m} (sl(n) )$ is a simple vertex algebra \cite{KacGorelik}.  But $\mbox{Im} (\Phi)$  need not to be simple.

\begin{example}  \label{ex-non-simp}  Let $m =  2n $. Since $V^{-2n} (sl(n))$ is a simple vertex algebra, we have
$$ \mbox{Im} (\Phi) = V_{-2n} (sl(n)) \otimes \mathcal W_{-n} $$
where 
$\mathcal W_{-n} $
 is a certain quotient of  $V^{-n} (sl(2n) )$. We claim that it is not simple. We noticed in the proof of Theorem \ref{thm1.1}(1)  that 
 $\mathcal W_{-n} $ 
 contains a module whose lowest component is isomorphic to $L_{sl(2n)} (\theta)$ (this module is realized inside of  $V_{-1}(sl(m n) )$).   But the classification of irreducible $V_{-n} (sl(2n))$--modules \cite{ArM-II}  shows that such module can not be  $V_{-n} (sl(2n))$--module. (It was proved in  \cite[Section 8]{ArM-II} that an irreducible $ V_{-n} (sl(2n) )$--module in the category $KL_{-n}$  must have lowest component isomorphic to $L_{sl(2n)} ( r \omega_n)$, for certain $r \in {\Z}_{\ge 0}$.)

This proves that  
$ V_{-n} (sl(2n) )  \ne  \mathcal W_{-n}$,
and therefore $V_{-1}(sl( 2 n^2 ) )$ is not a  semisimple $V^{-2n } (sl(n))  \otimes V^{-n} (sl(2 n ) )$--module.
\end{example}


  \section{Conformal embeddings of $\widetilde V(k,\k)$ in  $V_k(\g)$ with $\g$ of exceptional type.}

We will use the following notation: if $X$ denotes the type of a simple subalgebra $\k$ of a simple Lie algebra $\g$, then  $X^d$ identifies the embedding of Dynkin index $d$. We omit the superscript when $d=1$.
\par
Recall that the central charge of the Sugawara Virasoro vector for $V^k(\g), \g$ simple, is given by
$$c_k(\g)=\frac{ k \dim\g}{k+h^\vee}.$$
Also, if $\k\subset\g$ is a reductive subalgebra in $\g$, then we denote by $c(\k)$ the central charge of the Sugawara
Virasoro vector of  $\widetilde V(k,\k)$ in  $V_k(\g)$.
\begin{theorem}  Let $\k$ be a maximal non-equal rank semisimple subalgebra of an exceptional simple Lie algebra $\g$. Then  $\widetilde V(k,\k)$ is conformally embedded in  $V_k(\g)$ if and only if  either $k=1$ and $(\k,\g)$ belong to the following list
\begin{align}\label{listauno}
&(G_2\times F_4, E_8),\quad
(A_2^6\times A_1^{16}, E_8),\quad
(B_2, E_8),\\\notag
&(A_2^{21}, E_7),\quad
(G_2\times C_3, E_7),\quad
(G_2\times A_1^7, E_7),\quad
(F_4\times A_1^3,E_7),\\\notag
&(A_2^9, E_6),\quad
(G_2^3, E_6),\quad
(C_4, E_6),\quad
(G_2\times A_2^2, E_6),\\\notag
&(G_2\times A_1^8,F_4),\quad
(A_1^{28},G_2),\end{align}
or we are in one of the following cases:
\begin{align}\label{listadue1}
&\widetilde{V}(-6,G_2\times F_4)\hookrightarrow V_{-6}(E_8),\\\label{listadue2}&
\widetilde{V}(-4,F_4\times A_1^3)\hookrightarrow V_{-4}(E_7),\\\label{listadue3}&
\widetilde{V}(-3,G_2\times A_2^2)\hookrightarrow V_{-3}(E_6),\\\label{listadue4}&
\widetilde{V}(-3,F_4)\hookrightarrow V_{-3}(E_6),\\\label{listadue5}&\widetilde{V}(-5/2,G_2\times A_1^8)\hookrightarrow V_{-5/2}(F_4).
\end{align}
\end{theorem}
\begin{proof} We start from Dynkin's classification of non-equal rank maximal  subalgebras $\k$ in exceptional Lie algebras $\g$ (see Subsection \ref{classDyn}). For each pair $(\k=\oplus_i\k_i,\g)$ we determine the values 
of $k$ such that $c_k(\g)=\sum_i c(\k_i)$. We get either $k=1$ if $(\k,\g)$ appears in \eqref{listauno}, or the embeddings  displayed in \eqref{listadue1}--\eqref{listadue5},Ê
 or the triples $(\k,\g,k)$ displayed in the following list
\begin{align}\label{listatre}
&(A_1^{1240}, E_8,64/75),\quad
(A_1^{760}, E_8,8488/23275),\quad
(A_1^{520}, E_8,5788/15925),\\\notag
&(A_2^6\times A_1^{16},E_8,-119/474),\quad
(A_1^{389}, E_7,16/39),\quad
(A_1^{231}, E_7,872/2145),\\\notag
&(G_2^2\times A_1^{7},E_7,-26/29),\quad
(A_1^{24}\times A_1^{15},E_7,\frac{479\pm3 \sqrt{46265}}{1524}).
\end{align}\par 
The statements for $k=1$ are known (cf. \cite{AGO}, \cite{SW}, \cite{KS}). Case \eqref{listadue4} is treated in \cite{A}. For the other cases we use AP-criterion, verifying  condition \eqref{numcheck} in each of the cases. \par
Computing explicitly the decomposition via the package SLA \cite{sla}, it turns out that  for the levels $-1-h^\vee/6$ we have for \eqref{listadue5}
 \begin{equation}\label{d1}\p= L_{G_2}(\dot\theta_s)\otimes L_{A_1}(4\ddot\omega_1).\end{equation}
In cases \eqref{listadue1}, \eqref{listadue2}, \eqref{listadue3}, denoting the subalgebra in the leftmost part of each formula by $\dot\aa\times\ddot\aa$, we have
   \begin{equation}\label{d2}\p= L_{\dot \aa}(\dot\theta_s)\otimes L_{\ddot\aa}(\ddot\theta_s),\end{equation}
  where $\dot\theta_s$ (resp. $\ddot\theta_s$) is the highest short root of $\dot\aa$ (resp. $\ddot\aa$) and the convention that $\ddot\theta_s$ is the highest root if $\ddot\aa$ is simply laced.\par
  Decompositions \eqref{d1}, \eqref{d2} allow to check directly equality \eqref{numcheck}. Now we have to exclude the cases appearing in \eqref{listatre}.  We start from  the cases when $\k$ is of type $A_1$:  decompose $\p$ as a sum of $A_1$ modules; if $l$ is the dimension of a summand we should have, by \eqref{numcheck2}, that $\frac{l^2/2+l}{2(dk+1)}=1$. But a direct check shows that this is not possible, since the previous equation has no integral solution.\par
  If $\k$ is non-simple, then one computes $\p$ in  a computer assisted way and then checks that   \eqref{numcheck} does not hold. \end{proof}

 
 
 \section{The classification of conformal embeddings of maximally reductive subalgebras in $V_k(\g)$.}
 In this section we summarize our results and give the complete classification of conformal embeddings $\widetilde V(k,\k)\subset V_k(\g)$ with $\g$ a simple Lie algebra and $\k$ a semisimple maximal subalgebra of $\g$.
 
 
 The conformal embeddings with integrable $k$ (hence $k=1$) have been classified long ago: see \cite{AGO}, \cite{GNO}, \cite{KS}, \cite{SW}. We now give the classification of the conformal embeddings at non-integrable level. The following result summarizes the results of the previous sections together with the results of \cite{AKMPP1}.
 
 \begin{theorem}\label{global}
Assume that $\k$ is a maximally reductive subalgebra of a simple Lie algebra $\g$. Then there is a conformal embedding of $\widetilde V(k,\k)$ in $V_k(\g)$ with $k\ne 1$ only  in the following cases:
\begin{enumerate}
\item for $\g$  a simple Lie algebra of classical type and $\k$ semisimple
 \begin{align*}
&\widetilde{V}(-1,sl(n)\times sl(m))\hookrightarrow V_{-1}(sl(mn)),\ (m\ne n)\hbox to 10 cm{}\\
& \widetilde{V}(-\half,sp(2n)\times so(m))\hookrightarrow V_{-\half}(sp(2nm)), (m\ne 2n+2),\\
&\widetilde{V}(2-(n+m)/2,so(n)\times so(m))\hookrightarrow V_{2-(n+m)/2}(so(n+m)),(m\ne n)\\
&\widetilde{V}(-1/2,sp(2n)\times sp(2m))\hookrightarrow V_{-1/2}(sp(2(n+m))),\\
&\widetilde{V}(-1-(n+m)/2,sp(2n)\times sp(2m))\hookrightarrow V_{-1-(n+m)/2}(sp(2(n+m)),(m\ne n)\\
&\widetilde{V}(-1,sp(2n))\hookrightarrow V_{-1}(sl(2n)),\\&\widetilde{V}(-2,G_2)\hookrightarrow V_{-2}(so(7));\end{align*}
  \item for $\g$  an exceptional simple Lie algebra and $\k$ semisimple
\begin{align*}
&\widetilde{V}(-6,A_1\times E_7),\ \widetilde{V}(-6,A_2\times E_6),\ \widetilde{V}(-6,G_2\times F_4) \text{ in $V_{-6}(E_8)$},\hbox to 10 cm{}\\ 
&\widetilde{V}(-4,A_1\times D_6),\ \widetilde{V}(-4,A_2\times A_5),\ \widetilde{V}(-4,F_4\times A_1^3)\text{ in $V_{-4}(E_7)$},\\ 
&\widetilde{V}(-3,A_1\times A_5),\ \widetilde{V}(-3,G_2\times A_2^2),\ 
\widetilde{V}(-3,F_4)\text{ in $V_{-3}(E_6)$},\\
&\widetilde{V}(-5/2,A_1\times C_3),\  \widetilde{V}(-5/2,A_2\times A^2_2),\  
\widetilde{V}(-5/2,B_4),\ \widetilde{V}(-5/2,G_2\times A_1^8)\text{ in $V_{-5/2}(F_4)$},\\
&\widetilde{V}(-5/3,A_1\times A_1^3),\ 
\widetilde{V}(-5/3,A_2)\text{ in $V_{-5/3}(G_2)$};
\end{align*}
\item for $\k$ reductive non-semisimple ($Z$ denotes the one-dimensional center of $\k$)
\begin{align*}
&\widetilde{V}(-1,A_h\times A_{n-h-1}\times Z)\hookrightarrow V_{-1}(A_n), (h\ge 1,n-h\ge 2)\hbox to 10 cm{}\\
&\widetilde{V}(-(n+1)/2,A_h\times A_{n-h-1}\times Z)\hookrightarrow V_{-(n+1)/2}(A_n), (h\ge 1,h\ne(n-1)/2)\\
&\widetilde{V}(2-n,D_{n-1}\times  Z)\hookrightarrow V_{2-n}(D_n),\ \widetilde{V}(-2,A_{n-1}\times  Z)\hookrightarrow V_{-2}(D_n),\\
&\widetilde{V}(-1/2,A_{n-1}\times  Z)\hookrightarrow V_{-1/2}(C_n),\\
&\widetilde{V}(3/2-n,B_{n-1}\times  Z)\hookrightarrow V_{3/2-n}(B_n),\\
&\widetilde{V}(-3,D_5\times  Z)\hookrightarrow V_{-3}(E_6),\\
&\widetilde{V}(-4,E_6\times  Z)\hookrightarrow V_{-4}(E_7).
\end{align*}
 \end{enumerate}
 \end{theorem}
 \begin{proof}
We first discuss the $\k$ semisimple case. Combining the results of \cite{AKMPP1} and of Sections 3--5 above, one obtains the conformal embeddings listed in (1) and (2) except for the embedding $(B3)_{\omega_3}$, which is not listed. The reason is that $(B3)_{\omega_3}$ is a special case of the
$ \widetilde{V}(2-(n+m)/2, so(n)\times so(m))\hookrightarrow V_{2-(n+m)/2}(so(n+m))$ embedding, with $n=7$ and $m=1$. To check this claim we have to clarify that we consider two embeddings $\k\hookrightarrow\g$, $\k'\hookrightarrow\g$ to be equivalent if there is an automorphism of $\g$ mapping $\k$ to $\k'$. In such a case, in fact, there is an automorphism of $V_k(\g)$ that fixes the Virasoro vector and maps $\widetilde V(k,\k)$ onto $\widetilde V(k,\k')$. In the case at hand, let $\sigma$ be the automorphism of $so(8)$ induced by the diagram automorphism of order three that maps $\a_1$ to $\a_4$ (with respect to the usual Bourbaki numbering of simple roots). Let $\pi:so(8)\to gl(8)$ be the representation defined by $\pi(X)(v)=\sigma(X)(v)$. Since the defining representation of $so(8)$ is $L_{so(8)}(\omega_1)$, we see that $\pi$ is given by the action of $so(8)$ on $L_{so(8)}(\omega_4)$. It is easily checked that the restriction of  $L_{so(8)}(\omega_4)$ to $so(7)$ is $L_{so(7)}(\omega_3)$, which is the spin representation. It follows that the subalgebra of $so(8)$ corresponding to the action of $so(7)$ on the spin representation is $\pi(so(7))=\sigma(so(7))$.

By Lemma \ref{l1}, if $\k$ is not semisimple, then it is a maximally reductive equal rank subalgebra and conformal embeddings of such subalgebras are determined in \cite{AKMPP1}.
 \end{proof}

\section{On conformal embedding $ sl(2) \times so(m) $ into $sp (2m)$ at $k=-1/2$.}
\label{analysis-II}
Recall that the affine vertex algebra $V_{-1/2} (sp(2m))$ is realized as the even  subalgebra of the Weyl vertex algebra $M_{(m)}$. By using the conformal embedding of   $ sl(2) \times so(m) $ into $sp (2m)$ at level $k=-1/2$, we get an action of $\widehat{sl(2)} \times \widehat{so(m)}$ on $M_{(m)}$.  In this section we shall assume that $m \ne 4$ to exclude the  critical level case for $\widehat{sl(2)}$.

Since $V^{-m/2} (sl(2)) = V_{-m/2} (sl(2))$, we have 
   that $$\widetilde V(-1/2, sl(2)\times so(m))\cong  V_{-m/2} (sl(2)) \otimes  \widetilde V_{-2} (so(m))$$ where $ \widetilde V_r (\g)$ denotes a quotient (not necessarily simple) of $V^r (\g)$. We will use freely this notation in the following sections. Hence we have that 
$M_{(m)}$ becomes a $ V_{-m/2} (sl(2)) \otimes  \widetilde V_{-2} (so(m))$--module, and we are interested in its  decomposition. It is important to notice that this case is exactly an infinite-dimensional analog of the action of dual pair $sl(2) \times so(m)$ on the polynomial algebra ${\C}[z_1, \dots, z_m]$ studied  by R. Howe in \cite[Section 4]{H-Tams}. This classical result has very important applications  in  the representation theory. R. Howe proved that ${\C}[z_1, \dots, z_m]$ is a completely reducible $sl(2) \times so(m)$--module. In studying the same problem in the affine setting, we construct a family of singular vectors in $M_{(m)}$ which exactly correspond to Howe singular vectors in the decomposition of ${\C}[z_1, \dots, z_m]$. We study complete-reducibility problem. As a byproduct, we prove  that the affine vertex algebra $V_{-2}(D_3)= V_{-2}(A_3)$ is realized as a subalgebra of $M_{(6)}$. But we show  that in general (for  $m \ge 8$) $M_{(m)}$ is not completely reducible as   $V_{-m/2} (sl(2)) \otimes  \widetilde V_{-2} (so(m))$--module.

\vskip 5mm

We fix the following notation.
\begin{itemize}

\item Let $\widetilde V_{-2}(D_n) = \widetilde V_{-2}(so(2n) ) $ be the vertex subalgebra of $V_{-1/2} (C_{2n})$ generated by the factor $D_n$ in conformal embedding $ sl(2) \times D_n$ into $C_{2 n}$ at $k=-1/2$ $(n \ge 3)$.

\item Let $\overline V_{-2}(A_{n-1} )$ be the vertex subalgebra of $V_{-1} ( A_{2n-1} )$ generated by the factor $A_{n-1}$ in conformal embedding $ sl(2) \times A_{n-1}$ into $A_{2n-1}$ at $k=-1$ ($n \ge 3$).

\item Let $\widetilde V_{-2}(B_n) = \widetilde V_{-2}(so(2n+1) ) $ be the vertex subalgebra of $V_{-1/2} (C_{2n+1})$ generated by the factor $B_n$ in conformal embedding $ sl(2) \times B_n$ into $C_{2 n+1}$ at $k=-1/2$ $(n \ge 2)$. When  $n=1$, the subalgebra generated by the factor $B_n$ is isomorphic to $V_{-4} (sl(2))$.

 \item Let $M(1)$ be the Heisenberg vertex algebra of rank (and central charge) $1$. 

\end{itemize}

Recall also that $V_{-1/2} (sp(2m))$ can be realized as a subalgebra of the Weyl vertex algebra $M_{(m)}$ generated  by  bosonic fields $a_i ^{\pm}$ ($i=1, \dots, m$) with $\lambda$--brackets
$$ [ (a^{\pm} _i )_{\lambda}  (a^{\pm} _j ) ] = 0, \quad [ (a^{+} _i )_{\lambda}  (a^{-} _j ) ] = \delta_{i,j}. $$
We shall now consider $M_{(m)}$ as a module for $V_{-m/2} (sl(2)) \otimes \widetilde  V_{-2} (so(m))$.

\begin{lemma} \label{podmodul} 

\item[(1)] Assume that $m \ge 5$. 
Then there exist highest weight modules $\widetilde L_{\widehat{sl(2)}} (- (\frac{m}{2} + k ) \Lambda_0 +  k \Lambda_1), \widetilde  L_{\widehat{so(m)} } (- (2+  k  ) \Lambda _0 + k  \Lambda_1)$ in the category $KL$ 
such that 
$M_{(m)}$  contains  a $V_{-m/2} (sl(2)) \otimes \widetilde  V_{-2} (so(m))$--submodule isomorphic to
$$ M^{sub}=  \sum_{k=0} ^{\infty} \widetilde L_{\widehat{sl(2)}} (- (\frac{m}{2} + k ) \Lambda_0 +  k \Lambda_1) \bigotimes \widetilde  L_{\widehat{so(m)} } (- (2+  k ) \Lambda_0 +  k \Lambda_1).$$
 In particular,   $V_{-1/2}(sp(2m))$ contains  a $V_{-m/2} (sl(2)) \otimes \widetilde  V_{-2} (so(m))$--submodule  isomorphic to
$$ (M^{sub} )^0= \sum_{k=0} ^{\infty} \widetilde L_{\widehat{sl(2)}} (- (\frac{m}{2} +2 k  ) \Lambda_0 + 2 k \Lambda_1) \bigotimes \widetilde  L_{\widehat{so(m)} } (- (2+ 2 k  ) \Lambda_0 + 2 k \Lambda_1). $$

\item[(2)] Assume that $m =3$. Then $M_{(3)}$  contains  a $V_{-3/2} (sl(2)) \otimes  V_{-4} (sl(2) )$--submodule isomorphic to
$$ M^{sub}=  \sum_{k=0} ^{\infty}  L_{\widehat{sl(2)}} (- (\frac{3}{2} + k ) \Lambda_0 +  k \Lambda_1) \bigotimes   \widetilde L_{\widehat{sl(2)} } (- (4 +  2 k ) \Lambda_0 +    2 k \Lambda_1). $$
\item[] (Note that we don't claim  that sums above are direct!)
\end{lemma}
\vskip5pt
\vskip5pt



\begin{proof} (1) The explicit embedding of $V_{-1/2}(sp(2m))$ into $M_{(m)}$ is given by 
$$E_{m+j,m+i}-E_{i,j}\mapsto : a_i^+a_j^-:,\,E_{i,m+j}+E_{j,m+i}\mapsto : a_i^+a_j^+:,$$ $$
E_{m+i,j}+E_{m+j,i}\mapsto : a_i^-a_j^-:.$$ 

The embedding of $sl(2)\times so(m)$ into $sp(2m)$ is given by 
$$\left(\begin{pmatrix} a & b \\c & -a\end{pmatrix}, B\right)\mapsto \begin{pmatrix} a\cdot Id & b \cdot Id \\c\cdot Id & -a \cdot Id\end{pmatrix}+\begin{pmatrix} B  & 0 \\ 0 & B \end{pmatrix}.$$

As a Cartan subalgebra for $so(m)$ choose the block  matrices  of type 
\begin{equation}\label{CSAS}
{\begin{pmatrix} 0  & H  \\ - H  & 0\end{pmatrix}}\text{ if $m=2n$},\quad \begin{pmatrix} 0  & H&0  \\ - H  & 0&0\\ 0&0&0\end{pmatrix} \text{ if $m=2n+1$},
\end{equation} with  $H$ a diagonal $n\times n$ matrix. Root vectors in $so(2n)$
are matrices 
\begin{equation}\label{E}\mathcal E_\a=A-A^t,\end{equation} with $A=aE_{i,j}+ bE_{i,n+j}+c E_{n+i,j}+ d E_{n+i,n+j}$ and 
$$
\begin{pmatrix}a&b\\c&d\end{pmatrix}=\begin{cases}\begin{pmatrix}1&\sqrt{-1}\\-\sqrt{-1}&1\end{pmatrix}&\text{for $\a=\epsilon_i-\epsilon_j$,}\\[15pt]
\begin{pmatrix}1&-\sqrt{-1}\\-\sqrt{-1}&-1\end{pmatrix}&\text{for $\a=\epsilon_i+\epsilon_j$,}\\[15pt]
\begin{pmatrix}1&\sqrt{-1}\\\sqrt{-1}&-1\end{pmatrix}&\text{for $\a=-\epsilon_i-\epsilon_j$.}\\
\end{cases}
$$

The root vectors for $so(2n+1)$  are obtained by adding to the root vectors for $so(2n)$ constructed above the vectors $\mathcal E_{\epsilon_i}=A-A^t$ with  $A=E_{i,2n+1}-\sqrt{-1}E_{n+i,2n+1}$ and $\mathcal E_{-\epsilon_i}=A-A^t$ with  $A=E_{i,2n+1}+\sqrt{-1}E_{n+i,2n+1}$.\par
Consider  $\varphi = a_1  ^+ - \sqrt{-1} a_{n+1} ^+$. Note that  it has weight $\epsilon_1$. Indeed, let $h_r$ be the element of the Cartan subalgebra by substituting in \eqref{CSAS} $H=diag(0,\ldots,\sqrt{-1}, \ldots,0)$, with  the nonzero entry in the $r$-th position. The action of $h_r$ is given by  $-\sqrt{-1}(:a_{r}^+a^-_{n+r}:- :a_{n+r}^+a^-_{r}:)_{(0)}$, but
 $$[{\sqrt{-1}(:a_{n+r}^+a^-_r:- :a_{r}^+a^-_{n+r}:)}_{\lambda}\varphi]=-\delta_{r,1} i a_n^++\delta_{n+r,n+1} a_1^+=\d_{r,1}\varphi.$$
Now we study the action of ${\mathcal E_{\epsilon_1-\epsilon_2}}_{(0)},{\mathcal E_{-\epsilon_1-\epsilon_2}}_{(1)} $ on $\varphi$. We have that for any root $\a$, only the $(0)$-th product occurs in $[{\mathcal E_\a}_\lambda \varphi]$, 
hence ${(\mathcal E_\a)}_{(1)}\varphi=0$.  For ${\mathcal E_{\epsilon_1-\epsilon_2}}_{(0)}$ we have 
\begin{align*}[&{\mathcal E_{\epsilon_1-\epsilon_2}}_{\l}\phi]=
[(-:a^+_1a^-_2:-\sqrt{-1}: a^+_1a^-_{n+2}:+:a^+_2a^-_1-\sqrt{-1} :a^+_2a^-_{n+1}:)_{\lambda}\varphi]\\&+[(\sqrt{-1} :a^+_{n+1}a^-_2: - :a^+_{n+1}a^-_{n+2}:+\sqrt{-1} :a^+_{n+2}a^-_1:+:a^+_{n+2}a^-_{n+1}:)_{\lambda}\varphi]\\&=
-a^+_2+a^+_2- \sqrt{-1} a^+_{n+2}+\sqrt{-1}a^+_{n+2}=0.
\end{align*}

In particular
$$[{\mathcal E_{\epsilon_1-\epsilon_2}}_{(0)},\varphi_{(-1)}]=[{\mathcal E_{-\epsilon_1-\epsilon_2}}_{(1)},\varphi_{(-1)}]=0.$$ An obvious induction shows that $\varphi_{ (-1)}  ^k {\bf 1}$ is a singular vector for $\widehat{so(m)}$ of weight $2\L_0+k\epsilon_1=-(2+k)\L_0+k\L_1$.

The positive root  vector of $sl(2)$ corresponds to $\sum_{r=1}^m : a_r^+a_r^+:$ and the  $\lambda$--bracket  $[: a_r^+a_r^+:_\l\varphi]$ is trivial. 
Let  $\beta$ be the positive root of $sl(2)$,  let $h$ be corresponding coroot; then  $h$ corresponds to $-\sum_r:{a^+_ra^-_r:}$. 
Since
$
-\sum_r[:{a^+_ra^-_r:}_\lambda\varphi]=\varphi,
$
we see that the $sl(2)$--weight of $\phi$ is $\be/2$. Thus, arguing as for $so(m)$, we see that $(\phi_{(-1)})^k\vac$ is a singular vector for $\widehat {sl(2)}$ of weight $-\frac{m}{2}\L_0+(k/2)\be=-(\frac{m}{2}+k)\L_0+k\L_1$, hence it
generates a submodule isomorphic to
\begin{equation}\label{1} \widetilde L_{\widehat{sl(2)}} (- (\frac{m}{2} + k ) \Lambda_0 +  k \Lambda_1) \bigotimes \widetilde  L_{\widehat{so(m)} } (- (2+  k ) \Lambda_0 +  k \Lambda_1)  \quad (\mbox{for} \ m\ge 5).\end{equation}
%

 Consider now the case $m=3$. In this case the only positive root is $\epsilon_1$ whose root vector $\mathcal E_{\epsilon_1}$ maps to 
 $$
-:a^+_1a^-_3:+\sqrt{-1}a^+_2a^-_3:+:a^+_3a^-_1:-\sqrt{-1}:a^+_3a^-_2:.
$$
 and 
 $$[{\mathcal E_{\epsilon_1}}_\l \varphi]=-a^+_3+a^+_3=0.
 $$
 The $so(3)$--weight of $\phi$ is $\epsilon_1$.
 
We now perform the calculation of the level of the $\widehat{so(3)}$ action: the trace form of $sp(6)$ restricted to  $so(3)$ turns out to be $8(\cdot,\cdot)_{norm}$ where $(\cdot,\cdot)_{norm}$ is the normalized for of $so(3)$. It follows that the $so(3)$--level is $8(-\frac{1}{2})=-4$. Hence the $\widehat{so(3)}$--weight of $(\varphi_{(-1)})^k\vac$ is $-4\L_0+k\epsilon_1=-(4+2k)\L_0+2k\L_1$.
The computation of the $\widehat{sl(2)}$--weight of $(\varphi_{(-1)})^k\vac$ is the same as in the previous case. It follows that $(\varphi_{(-1)})^k\vac$ is a singular vector generating 
$$L_{\widehat{sl(2)}} (- (\frac{3}{2} + k ) \Lambda_0 +  k \Lambda_1) \bigotimes   \widetilde L_{\widehat{sl(2)}  } (- (4 +  2 k  ) \Lambda_0 +   2 k \Lambda_1).$$
The irreducibility of $V_{-3/2} (sl(2)) . (\varphi_{(-1)})^k\vac   \cong L_{\widehat{sl(2)}} (- (\frac{3}{2} + k ) \Lambda_0 +  k \Lambda_1)  $  follows from the fact that the Weyl module in $KL_{-3/2}$ with the highest weight  $- (\frac{3}{2} + k ) \Lambda_0 +  k \Lambda_1$ is irreducible  (see Section \ref{m3}). 
\end{proof}
\begin{rem}\label{for11}
It is interesting to investigate whether $M_{(m) } = M^{sub}$,  $V_{-1/2}(sp(2m)) = (M^{sub}) ^0$.  In Section \ref{m3}, we prove these equalities hold  in the case $m=3$. But  in Section \ref{m8} we show that in the case $m=8$ there is a subsingular vector $P^-_{high}$ outside of  $M^{sub}$.  A possible reason for a difference between vertex-algebraic and classical setting is that subsingular vectors belong to the kernel of Zhu's functor (cf. Remark \ref{observation}).
\end{rem}

We have inclusions $A_{n-1} \hookrightarrow D_n \hookrightarrow sp(4n)$. We let 
$\overline V_{-2} (A_{n-1})$ be  the vertex subalgebra of $V_{-1/2}(sp(4n))$ generated by 
the elements in $A_{n-1}$.

\begin{theorem} \label{thm-simplicity}
\item[(1)] $\overline  V_{-2}(A_{n-1} ) \otimes M(1) $  is conformally embedded into $\widetilde V_{-2}(D_n)$  for $n \ge 3$.

\item[(2)]  $\widetilde V_{-2}(D_n)$ is simple if and only if $n=3$.

\item[(3)] $\widetilde V_{-2}(B_n)$  is not simple for $n \ge 3$.
\end{theorem}
\begin{proof}
The proof of assertion (1) follows from the following observations:
\begin{itemize}
\item[(a)] $\widetilde{V}_{-1} ( A_{2n-1} ) \otimes M(1)$ is conformally embedded into $V_{-1/2} (C_{2n} )$ and 
$\widetilde{V}_{-1} ( A_{2n-1} ) \otimes M(1)= V_{-1} ( A_{2n-1} ) \otimes M(1)$
\cite[Theorem 5.1, (4)]{AKMPP1};
\item[(b)] by Theorem \ref{thm1.1}, $V_{-n}(sl(2)) \otimes \overline  V_{-2}(A_{n-1})$  is conformally embedded into $V_{-1} ( A_{2n-1} )$ ;
\item[(c)] $V_{-n}(sl(2)) \otimes \widetilde V_{-2}(D_n )$ is conformally embedded into $V_{-1/2} ( C_{2n} )$: see Theorem \ref{thm1.1} (4) in case $n=1$.
\end{itemize}
Start with the following inclusion of vertex algebras:
\begin{equation}\label{one}\overline  V_{-2}(A_{n-1}) \otimes M(1) \hookrightarrow \widetilde V_{-2}(D_n).\end{equation} Tensoring both members of \eqref{one} with $V_{-n}(sl(2))$,  we obtain a chain:
\begin{equation}\label{two} V_{-n}(sl(2))\otimes\overline  V_{-2}(A_{n-1}) \otimes M(1) \hookrightarrow V_{-n}(sl(2))\otimes\widetilde V_{-2}(D_n)\to V_{-1/2} ( C_{2n} ),\end{equation}
where, by (c), the rightmost map is a conformal embedding. So, to prove (1), it suffices to prove that the embedding $V_{-n}(sl(2))\otimes\overline  V_{-2}(A_{n-1}) \otimes M(1)\to V_{-1/2} ( C_{2n} )$ is conformal. This follows from the chain
$$  V_{-n}(sl(2))\otimes\overline  V_{-2}(A_{n-1}) \otimes M(1) \to  V_{-1} ( A_{2n-1} ) \otimes M(1) \to V_{-1/2} ( C_{2n} ),
$$
where the leftmost map is conformal by (b) and the rightmost is conformal by (a). 
\par
We now prove statements (2) and (3).  Recall that  vertex algebras $V_{-2}(D_n  )$ (for $n \ge 4$)   
and $V_{-2} (B_n)$ (for $n \ge 3$)
  have  only finitely many irreducible modules in the category $\mathcal O$ \cite{ArM-I}, \cite{ArM-III}. On the other hand, Lemma \ref{podmodul} shows    that   $ \widetilde  V_{-2}(D_n  )$  
   and $\widetilde V_{-2} (B_n)$
   have infinitely many irreducible modules, and therefore we  conclude that  $ \widetilde  V_{-2}(D_n  )$  (for $n \ge 4$)    
    and $\widetilde V_{-2} (B_n)$  (for $n \ge 3$) 
    cannot be simple.

The  simplicity of $\widetilde V_{-2}(D_3)$  follows from the following facts:
\begin{itemize}
\item The maximal ideal in $V^{-2}(D_3)$ is generated by a unique singular vector $v_{sing}$ of conformal weight $2$ (the vector  $\sigma (v_0)$ from \cite[Theorem 8.2]{ArM-II}, in the case $A_3= D_3$).
\item By \cite[[Theorem 0.2.1]{KacGorelik}, $ \overline V_{-2}(A_{2 }) = V_{-2}(A_{2 }) $. Since  $ V_{-2}(A_{2 }) \otimes M(1) $  is conformally embedded into $\widetilde V_{-2}(D_3)$, then  $v_{sing} = 0$ in $\widetilde V_{-2}(D_3)$.
Indeed, it was proved in \cite{A} that  in the case of conformal embeddings, a singular vector at conformal weight 2  must vanish. Here we have only one singular vector of such conformal weight.  
 \end{itemize}
 So $\widetilde V_{-2}(D_3) =  V_{-2}(D_3)$.
 \end{proof}

\begin{rem}  In our forthcoming paper \cite{AKMPP-2018} we prove that the simple vertex algebra $V_{-2} (B_2)$ embedds  into $M_{(5)}$.   We expect that in the cases $m=5,6$ $M_{(m)}$ is completely reducible.  But for $m \ge 7$ the vertex algebras   $\widetilde V_{-2}(so(m))$ are non-simple,  having a simple maximal ideal (cf. \cite{AKMPP-2018}). \end{rem}


\begin{rem}
Since $V_{-1/2} (C_{n} )$ can be realized using Weyl vertex algebra $M_{(n)}$, the previous theorem gives an explicit bosonic realization of the simple vertex algebra $V_{-2}(A_3) = V_{-2} (D_3)$.  Our result gives a chiralization of  \cite[Theorem 8.13]{ArM-II}, where  it was proved that there is an embedding
of Zhu's algebra $A (V_{-2}(A_3)) $ into the Weyl algebra.  

The conformal embedding in Theorem \ref{thm-simplicity} (1) gives a proper framework for studying conformal embedding $A_{n-1} \times Z $ in $D_n$ at $k=-2$ for which decomposition is still unknown in the cases $n=3,4$.


\end{rem}

We shall now describe branching rules for  conformal embedding $A_{n-1} \times Z $ in $D_n$  realized inside  $V_{-1/2} (C_{2n} )$. For a proof we need an important observation which gives a refinement of  \cite[Theorem 2.4]{AKMPP1}.

\begin{proposition} \label{refine} Assume that an affine vertex algebra $\overline V_k (\g_0)$  is conformally embedded into  $\widetilde V_k (\g)$ and $\widetilde V_k (\g)$ is a vertex subalgebra of  a simple vertex algebra $\mathcal U$.  In the hypothesis of  \cite[Theorem 2.4]{AKMPP1}, we have
$$ \widetilde V_k (\g) = \bigoplus _{q \in {\Z}} \widetilde V_k (\g) ^{(q)},   \quad  0 \ne  \widetilde V_k (\g) ^{(q)} \cdot  \widetilde V_k (\g) ^{(r)}   \subset \widetilde V_k (\g) ^{(q+r)} \quad (q,r \in {\Z}),$$
and each $ \widetilde  V_k (\g) ^{(q)}$ is a cyclic $\overline V_k (\g_0)$--module (i.e., highest weight $\widehat{\g}_0$--module).
\end{proposition}

\begin{cor} \label{cor-branching}
Assume that $m \ge 5$. Then
$$  \widetilde V_{-2} (D_m) = \sum_{ \ell \in {\Z} } \widetilde V_{-2} (D_m) ^{(\ell)} $$
and each  $\widetilde V_{-2} (D_m) ^{(\ell)} $ is a highest weight $\widehat{gl(m)}$--module at level $-2$.
\end{cor}
\begin{proof}
In the case $m \ge 5$, we proved in \cite[Theorem 5.1]{AKMPP1} that $$V_{-2} (D_m) = \sum_{ \ell \in {\Z} } V_{-2} (D_m) ^{(\ell)},$$ and each  $V_{-2} (D_m) ^{(\ell) }$ is an irreducible $V_{-2} (gl(m))$--module. Our proof used a fusion rules method. We shall now extend this result to a non-simple vertex algebra $\widetilde V_{-2} (D_m)$. 
Note that  $\widetilde V_{-2} (D_m)$ is realized as a subalgebra of the simple vertex algebra $M_{(2 m)}$.
As in  the  proof of \cite[Theorem 5.1 (3)]{AKMPP1}   we conclude that the conditions of \cite[Theorem 2.4]{AKMPP1} hold for $m \ge 5$. Now Proposition  \ref{refine} implies that
\bea  \widetilde V_{-2} (D_m) = \sum_{ \ell \in {\Z} } \widetilde V_{-2} (D_m) ^{(\ell)},\eea 
and each  $\widetilde V_{-2} (D_m) ^{(\ell) }$ is a highest weight    $\overline V_{-2} (A_{m-1} ) \otimes M(1)$--module. Moreover, $V_{-2} (D_m) ^{(\ell) }$ is a simple quotient of $\widetilde V_{-2} (D_m) ^{(\ell) }$. 
\end{proof}

 \section{ Howe dual pairs and the  Linshaw-Schwarz-Song's  method}\label{LSS}

In this section we shall  combine results from previous Section  and the methods from  the paper  \cite{LSS}. As a consequence we will get a new  realization of Howe dual pairs of affine vertex algebras at negative levels.

We now recall the setting of \cite{LSS}. Given a vector space $V$, denote by $\mathcal S(V)$ the vertex algebra with even generators $\beta^x$, $\gamma^{x'}$, $x\in V$, $x'\in V^*$ with $\lambda$--brackets
$$
[\be^x{}_\l \gamma^{x'}]=x'(x),\ [{\be^x}_\l \be^{y}]=0,\ [\gamma^{x'}{}_\l \gamma^{y'}]=0.
$$
If $V$ is a module for a reductive Lie algebra $\g=\sum_i \g_i$, then one can define a vertex algebra map $\hat\tau: \otimes_iV^{-k_i}(\g_i)\to \mathcal S(V)$ by setting 
$$
\hat\tau(X)=-\sum_{i} :\gamma^{x'_i}\be^{X\cdot x_i}:,
$$
where $\{x_i\}$ is a basis of $V$ and $\{x'_i\}$ is its dual basis, and $k_i$ is the ratio between the trace form induced by $V$ and the normalized invariant form of $\g_i$.
We specialize to $V=\C^2\otimes \C^m$ and $\g=gl(2)\times gl(m)$. 

Given a pair $(U,\langle\cdot,\cdot\rangle)$, where $U$ is a vector space and  $\langle\cdot,\cdot\rangle$ is a symplectic form, denote by  $M(U)$  the universal vertex algebra with generators $u\in U$ and $\l$--bracket defined by
$
[u_\l v]=\langle u,v \rangle
$.
Note that choosing $U=\C^{2m}$ with the standard symplectic form given by $\langle e_i,e_j\rangle =\d_{j,i+m}$ for $i=1,\ldots,m$; $j=1,\ldots,2m$, one has that the map $e_i\mapsto a^+_i$, $e_{i+m}\mapsto a^-_i$, $i=1,\ldots,m$, gives an isomorphism between $M(\C^{2m})$ and $M_{(m)}$. On the other hand, choosing $U=V\oplus V^*$ with the symplectic form such that $V$ and $V^*$ are isotropic and $\langle x, x'\rangle=x'(x)$ for $x\in V$ and $x'\in V^*$, one obtains that $\mathcal S(V)=M(V\oplus V^*)$.

Given  pairs $(U,\langle\cdot,\cdot\rangle)$, $(U',\langle\cdot,\cdot\rangle')$ then any linear isomorphism $\phi:U\to U'$ preserving the symplectic structures induces an isomorphism of vertex algebras between $M(U)$ and $M(U')$. In particular we define
$\phi: (\C^2\otimes \C^m)\oplus (\C^2\otimes \C^m)^*\to \C^{4m}$ by setting
$$
e_1\otimes e_j\mapsto \frac{1}{\sqrt{2}}(e_j+\sqrt{-1} e_{m+j}),\ e_2\otimes e_j\mapsto \frac{1}{\sqrt{2}}(e_{2m+j}+\sqrt{-1} e_{3m+j}),
$$
$$e^1\otimes e^j\mapsto \frac{1}{\sqrt{2}}(e_{2m+j}-\sqrt{-1} e_{3m+j}),\ e^2\otimes e^j\mapsto -\frac{1}{\sqrt{2}}(e_j-\sqrt{-1} e_{m+j}),\ 
$$  

The symplectic isomorphism $\phi$ induces an isomorphism 
$$\Phi:\mathcal S(\C^2\otimes \C^m)\to M_{(2m)}.
$$

We now compute $\Phi(\hat\tau( V^{-m}(sl(2))\otimes  V^{-2}(sl(m))))$. We claim that
$$\Phi(\hat\tau( V^{-m}(sl(2))\otimes  V^{-2}(sl(m))))=V_{-m}(sl(2))\otimes \ov V_{-2}(A_{m-1}).
$$
It is enough to compute the image under $\Phi\circ \hat\tau$ on generators of $sl(2)\times sl(m)$:

$$
\Phi(\hat\tau(e))=\Phi(-\sum_{i=1}^m:\gamma^{e^2\otimes e^i}\be^{e_1\otimes e_i}:)=-\frac{1}{2}\sum_{i=1}^{2m}:a^+_ia^+_i:.
$$
$$
\Phi(\hat\tau(h))=\Phi(-\sum_{i=1}^m(:\gamma^{e^1\otimes e^i}\be^{e_1\otimes e_i}:-:\gamma^{e^2\otimes e^i}\be^{e_2\otimes e_i}:)=-\sum_{i=1}^{2m}:a^+_ia^-_i:.
$$
$$
\Phi(\hat\tau(f))=\Phi(-\sum_{i=1}^m:\gamma^{e^1\otimes e^i}\be^{e_2\otimes e_i}:)=-\frac{1}{2}\sum_{i=1}^{2m}:a^-_ia^-_i:.
$$
If $i\ne j$, with notation as in \eqref{E}
\begin{align*}
\Phi(\hat\tau(E_{ij}))&=\Phi(-\sum_{r=1}^2 :\gamma^{e^r\otimes e^j}\be^{e_r\otimes e_i}:)=:a^+_ia^-_j:+\sqrt{-1}:a^+_{m+i}a^-_j:\\&-\sqrt{-1}:a^+_{i}a^-_{m+j}:+:a^+_{m+i}a^-_{m+j}:-:a^+_{j}a^-_{i}:-\sqrt{-1}:a^+_{j}a^-_{m+i}:\\&+\sqrt{-1}:a^+_{m+j}a^-_{i}:-:a^+_{m+j}a^-_{m+i}:=(\mathcal E_{\epsilon_j-\epsilon_i})_{(-1)}\vac.
\end{align*}
Finally, if $i=j$,
\begin{align*}
\Phi(\hat\tau(E_{ii}))&=\Phi(-\sum_{r=1}^2 :\gamma^{e^r\otimes e^i}\be^{e_r\otimes e_i}:)=\sqrt{-1}(:a^+_ia^-_{m+i}:-:a^+_{m+i}a^-_i:)=-(h_i)_{(-1)}\vac,
\end{align*}
hence $\Phi(\hat\tau(E_{ii}-E_{jj}))\in \ov V_{-2}(A_{m-1})$. 

As $gl(m)=\C I \oplus sl(m)$, we denote by $V^{-2}(gl(m))$ the vertex algebra $V^{-2}(sl(m))\otimes M(1)$. Note that the above computations show that $\Phi(\hat \tau(V^{-2}(gl(m))))\subset \widetilde V_{-2}(D_m)$.

\begin{proposition} \label{g-1} For $m \ge 3$ we have
\begin{equation}\label{com}Com(V_{-m } (sl(2)), M_{(2 m)} ) = \widetilde V_{-2} (so(2 m)).\end{equation}
\end{proposition}
\begin{proof} It is clear that $Com(V_{-m } (sl(2)), M_{(2 m)} ) \supset \widetilde V_{-2} (so(2 m))$. By   \cite[Theorem 4.3]{LSS}, $\mathcal S(\C^2\otimes \C^m)^{sl_2[t]}$ is generated by $\hat\tau(V^{-2}(gl(m)))$ together with  generators 
\begin{align*}D_{i,j}&=:\be^{e_1\otimes e_i}\be^{e_2\otimes e_j}:-:\be^{e_2\otimes e_i}\be^{e_1\otimes e_j}:, \ &&\{i,j\}\subset\{1,\ldots,m\},\\
D_{i,j}'&=:\gamma^{e^1\otimes e^i}\gamma^{e^2\otimes e^j}:-:\gamma^{e^2\otimes e^i}\gamma^{e^1\otimes e^j}:,\ &&\{i,j\}\subset\{1,\ldots,m\}.
\end{align*}

We already saw that $\Phi(\hat\tau(V^{-2}(gl(m))))\subset \widetilde V_{-2} (so(2 m))$. As for $\Phi(D_{i,j}),\Phi(D'_{i,j})$ we have
\begin{align*}
&2\Phi(D_{i,j})=:a^+_ia^-_j:+\sqrt{-1}:a^+_{i}a^-_{m+j}:+\sqrt{-1}:a^+_{m+i}a^-_{j}:-:a^+_{m+i}a^-_{m+j}:\\&-:a^+_{j}a^-_{i}:-\sqrt{-1}:a^+_{j}a^-_{m+i}:-\sqrt{-1}:a^+_{m+j}a^-_{i}:+:a^+_{m+j}a^-_{m+i}:=(\mathcal E_{-\epsilon_j-\epsilon_i})_{(-1)}\vac.
\end{align*}
Likewise
\begin{align*}
&-2\Phi(D'_{i,j})=:a^+_ja^-_i:-\sqrt{-1}:a^+_{m+j}a^-_{i}:-\sqrt{-1}:a^+_{j}a^-_{m+i}:-:a^+_{m+j}a^-_{m+i}:\\&-:a^+_{i}a^-_{j}:+\sqrt{-1}:a^+_{i}a^-_{m+j}:+\sqrt{-1}:a^+_{m+i}a^-_{j}:+:a^+_{m+i}a^-_{m+j}:=(\mathcal E_{\epsilon_i+\epsilon_j})_{(-1)}\vac.
\end{align*}
\end{proof}

\begin{cor} \label{cor-11} The following vertex subalgebras form   Howe dual pairs:
\begin{itemize}
\item[(1)] $V_{-m} (sl(2))$ and $\widetilde V_{-2} (D_m)$ inside  $V_{-1/2} (C_{2m})$ for $m \ge 3$.
\item[(2)] $V_{-m} (sl(2))$ and $\overline V_{-2} (A_{m-1})$ inside  $V_{-1} (A_{2m-1})$ for $m \ge 5$.
\end{itemize}

\end{cor}
\begin{proof}
The proof of assertion (1) follows from Proposition \ref{com}.

 By using Corollary  \ref{cor-branching}  we get
\bea  \widetilde V_{-2} (D_m) = \sum_{ \ell \in {\Z} } \widetilde V_{-2} (D_m) ^{(\ell)},\eea 
and each  $\widetilde V_{-2} (D_m) ^{(\ell) }$ is a highest weight    $\overline V_{-2} (A_{m-1} ) \otimes M(1)$--module. So $V_{-m} (sl(2))$ and $\overline V_{-2} (A_{m-1} ) \otimes M(1)$ form a Howe pair inside the charge zero component of $M_{(2m)}$, which,  by \cite{AP-JAA}, is  isomorphic to $V_{-1} (A_{2m-1}) \otimes M(1)$. This proves assertion (2).
\end{proof}

\begin{rem} \label{rem-81} Note that Corollary \ref{cor-11} in the case $m=4$   partially proves the   Conjecture (5.3) of D. Gaiotto from his recent paper  \cite{Ga}.  The conjecture is written as
$$ \frac{Sb( {\C} ^{16} )}{\widehat {SU}(2)_{-4} } \cong \widehat{SO}(8)_{-2}. $$
Since,  in Gaiotto's notation,  $Sb( {\C} ^{16})$ is precisely our $M_{(8)}$, the coset is indeed the  affine vertex algebra $\widetilde V_{-2} (D_4)$, but it is not simple (as is probably expected in \cite{Ga}).
\end{rem}


 


\section{The decomposition in the case $m=3$ and an $q$-series identity}

\label{m3}

 In this section we prove the complete reducibility of $M_{(3)}$ as $V_{-3/2} (sl(2)) \otimes V_{-4} (sl(2))$--module and find the explicit decomposition. It is interesting that this decomposition gives a vertex-algebraic interpretation of the  $q$--series identity from \cite[Example 5.2]{KW-1994} (see also  \cite[Theorem 4]{Ono}). This suggests that other  decompositions from the previous section are related to certain $q$--series  identities.


Let 
$$\phi(q) =  \prod_{n \ge 1} (1-q^n).$$ 
We shall first identify characters of $V_k(sl(2))$--modules for $k =-3/2$ and $k =-4$.

\begin{lemma} \label{karakteri} Let $ \ell \in {\Z_{\ge 0}}$. We have:
\begin{align*} \mbox{ch}_q  L_{\widehat{sl(2)}} (- (\frac{3}{2} + \ell  ) \Lambda_0 +  \ell  \Lambda_1) &=  q^{3/8}   \phi(q)  ^{-3}  (\ell +1) q^{\frac{\ell  ( \ell+2)}{2} } ,   \\
\mbox{ch}_q   L_{\widehat{sl(2)}} (- (4 +  2 \ell  ) \Lambda_0 +    2 \ell \Lambda_1) &= q^{-1/4}    \phi(q) ^{-3} \sum_{i=0}^ {\ell} (-1) ^{\ell -i}  (2 i +1) q^{-\frac{i  (i +1)}{2} }. 
\end{align*}
\end{lemma}
\begin{proof}
For $ r \in {\Z}_{\ge 0}$ and $k \in {\C}$  denote by $V^k (r\omega_1) $ the generalized Verma module
$$ V^k (r \omega_1) = U(\widehat{sl(2)}) \otimes _{U (\mathfrak p) }  L_{sl(2)} (r\omega_1),$$
where  $\mathfrak p = sl(2)\otimes {\C}[t] + {\C} K$    and $L_{sl(2)} (r\omega_1)$ denotes the  irreducible $(r+1)$--dimensional $sl(2)$--module, regarded as $\mathfrak p$--module on which $K$ acts by  $k \mbox{Id}$  and $sl(2)\otimes t {\mathbb C} [t]$ acts trivially.\par
The proof of the lemma is a consequence of  the following facts  from the structure theory of Verma modules for $\widehat{sl(2)}$ (we omit details):
\begin{itemize}
\item[(1)] $V^{-3/2} (\ell \omega_1)$ is irreducible for every $\ell \ge 0$, hence $$V^{-3/2} (\ell\omega_1)=  L_{\widehat{sl(2)}} (- (\frac{3}{2} + \ell  ) \Lambda_0 +  \ell  \Lambda_1).$$
One can prove  the irreducibility by using the fact  that  the Hamiltonian reduction maps the Weyl module $V^{-3/2} (\ell \omega_1)$  to  an  irreducible module for the Virasoro algebra at central charge $c=-2$. We omit details, since they are similar  to the proof of  \cite[Theorem 5.3]{AKMPP-JJM}.

\item[(2)] The vector space of all singular vector in $V^{-4} (2 \ell \omega_1)$ is spanned by the set 
$ \{ v_{i} \ \vert \  i = 0, \dots, \ell \} $
where $v_{i}$ is the unique (up to a constant) singular vector of $\g$--weight $2i  \omega_1$.
\item[(3)] If $\ell \ge 1$, the maximal submodule of $V^{-4} (2 \ell \omega_1)$  is irreducible, and it is generated by the singular vector $v_{\ell-1}$.
\end{itemize}
\end{proof}

\begin{theorem}   \label{m3-dec} $M_{(3)}$ is a completely reducible $V_{-3/2} (sl(2)) \otimes V_{-4}(sl(2))$--module and the following decomposition holds
\bea  && M_{(3) }=  \bigoplus_{\ell =0} ^{\infty}\left( L_{\widehat{sl(2)}} (- (\frac{3}{2} + \ell  ) \Lambda_0 +  \ell  \Lambda_1) \bigotimes    L_{\widehat{sl(2)}} (- (4 +  2 \ell  ) \Lambda_0 +    2 \ell \Lambda_1)   \label{m3-1} \right).\eea
\end{theorem}

\begin{proof}
Note that the $q$ character of $M_{(3)}$ is
\bea
\mbox{ch}_q   M_{(3)} &= & q^{-c/24 } \prod_{ n= 1} ^{\infty} (1-q^{n-1/2} )  ^{-6} \nonumber \\
                                   &= & q^{1/8 } \prod_{ n= 1} ^{\infty} (1-q^{n-1/2} )  ^{-6} \  \qquad ( \mbox{since} \ c =-3) \nonumber \\
  &= & q^{1/8}  \left( \frac{ \phi(q)}  {\phi(q^{1/2} )} \right) ^6    
= \ q^{1/8} \frac{1}{\phi(q) ^ 6} \left( \frac{ \phi(q)^2 }  {\phi(q^{1/2} )} \right) ^6     \nonumber \\
  & = & q^{1/8}  \frac{1}{\phi(q) ^ 6}  \Delta  (q^{1/2}) ^ 6 \nonumber 
\eea
where
$$\Delta  (q) = \sum_{ n \in {\Z}_{\ge 0} } q^ { n (n+1) /2} =\frac{  \phi(q^2 ) ^2 }{ \phi(q)}. $$

Using  Lemma  \ref{karakteri}  we get that the $q$-character  of the right side of (\ref{m3-1}) is given by
$$ \frac{q^{1/8} }{\phi(q) ^6} \sum_{\ell=0} ^{\infty} \sum_{i=0} ^{\ell}  ( -1) ^{\ell -i}  (\ell+1) (2 i +1) q^{\frac{ \ell (\ell +2) -i  (i +1)}{2} }.  $$
Therefore  the  proof of the theorem is now reduced to  the following  identity:
\bea  \label{identity} &&  \frac{ \phi(q) ^ {12} } { \phi(q^{1/2}) ^6 } = \sum_{\ell=0} ^{\infty} \sum_{i=0} ^{\ell}  ( -1) ^{\ell -i}  (\ell+1) (2 i +1) q^{\frac{ \ell (\ell +2) -i  (i +1)}{2} }.  \eea

Now we shall see that \eqref{identity} follows from 
the Kac-Wakimoto identity \cite[Example 5.2]{KW-1994}:
$$ \Delta(q) ^6 = -\frac{1}{8} \sum_{(j,k) \in S }  (-1) ^ {\frac{1}{4} (j-1) (k+1)}    (j^ 2 - k^2) q^{\frac{1}{4} (j k -3)},$$
where
 $$S =\{ (j,k)  \ \vert \  j, k , \frac{1}{2} (j -k) \in 2 {\Z}_{\ge 0}  +1, \  j > k \ge 1 \}. $$
For every $(j,k) \in S$, one can see that there are unique $ \ell, i \in {\Z}_{\ge 0} $, $i \le \ell$ such that
$$ j = 2\ell + 2i +  3, \quad k = 2\ell - 2i +1. $$
Hence we have:
\bea  \Delta(q) ^6  &=&  -\frac{1}{8} \sum_{(j,k) \in S }  (-1) ^ {\frac{1}{4} (j-1) (k+1)}    (j^ 2 - k^2) q^{\frac{1}{4} (j k -3) } \nonumber \\
 &=&(  -\frac{1}{8} ) \cdot   2 \cdot 4  (-1) \sum_{\ell=0} ^{\infty} \sum_{i=0} ^{\ell}  ( -1) ^{\ell -i}  (\ell+1) (2 i +1) q^{ \ell (\ell +2) -i  (i +1) }  \nonumber  \\
 &=& \sum_{\ell=0} ^{\infty} \sum_{i=0} ^{\ell}  ( -1) ^{\ell -i}  (\ell+1) (2 i +1) q^{ \ell (\ell +2) -i  (i +1) },  \nonumber   
 %
%
\eea
proving \eqref{identity}.
\end{proof}

\section{ The case  $m=8$ and  an application to conformal embeddings}
\label{m8}
As anticipated in Remark \ref{for11}, we now show that in the case $m=8$ there is a subsingular vector $P^- _{high}$ outside of  $M^{sub}$.   As a consequence,  the decomposition in this case is completely different from    the one in the  classical  (i.e., non-affine)  case. We believe that   a similar pattern will  occurr when $m \ge 8$. \par
We have the following facts.
\begin{itemize}
\item The vertex algebra $M_{(8)}$ is a module for $\widetilde V_{-2} (D_4) \otimes V_{-4} (sl(2))$.
\item $\widetilde V_{-2} (D_4) $ is non-simple, and 
$$\mbox{Com} ( V_{-4} (sl(2)), M_{(8)} ) \supset \widetilde V_{-2} (D_4).$$

\item $( \varphi_{(-1)} ) ^  2 {\bf 1}$ is a singular vector which generates a $\widetilde V_{-2} (D_4) \otimes V_{-4} (sl(2))$--submodule
$ \widetilde L_{\widehat{sl(2)}}  (- 6 \Lambda_0 +2 \Lambda_1) \otimes \widetilde L_{\widehat{so(8)}}  (- 4\Lambda_0 +2 \Lambda_1). $ So
$$ ( \varphi_{(-1)}  )^ 2 {\bf 1} = \widetilde w_{1} \otimes \widetilde w_{2 }, $$
where $\widetilde w_{1}$  (resp. $\widetilde w_{2})$  is a highest weight vector in $  \widetilde L_{\widehat{sl(2)}}  (- 6 \Lambda_0 +2 \Lambda_1) $ (resp. $ \widetilde L_{\widehat{so(8)}}  (- 4\Lambda_0 +2 \Lambda_1)$).

\item $ \widetilde L_{\widehat{sl(2)}}  (- 6 \Lambda_0 +2 \Lambda_1) $ is a certain quotient of the generalized Verma $\widehat {sl(2)}$--module $V^{-4} (2 \omega_1)$ which has a singular vector 
$$v_0=\left(  f(-1)  - \frac{1}{2} h(-1) f(0)   - \frac{1}{2} e(-1) f(0)^2 \right) v_{2 \omega_1} $$ such that
$$ V_{-4} (sl(2)) = U(\widehat {sl(2)}). v_0 \subset V^{-4} (2 \omega_1). $$
 So the vertex algebra $ V_{-4} (sl(2)) $ is embedded into  $V^{-4} (2 \omega_1)$.

\end{itemize}
\vskip10pt
Set  $n=m/2 = 4$ and define
\bea
& b_i ^ + = \frac{ a^+_ i  - \sqrt{-1} a^+_{n+i} }{\sqrt{2}},    & b_i ^ - = \frac{ a^- _ i  +  \sqrt{-1} a^-_{n+i} }{\sqrt{2}}, \nonumber \\
    & b_{n+i} ^ +  = \frac{ a^- _ i  -  \sqrt{-1} a^-_{n+i} }{\sqrt{-2}},& b_{n+i}  ^ - = \frac{ a^+_ i  + \sqrt{-1} a^+_{n+i} }{\sqrt{-2}}. \nonumber 
\eea
Then we have
$$ [ ( b_i ^{\pm})  _{\lambda} b_j ^{\pm}] = 0, \quad [ (b_i ^{+ } )  _{\lambda} ( b_j ^{-} )] = \delta_{i,j}. $$

Define the following vectors:
\bea
P^+ & = &    (  f(-1)   - \frac{1}{2} h(-1) f(0)  - \frac{1}{2} e(-1) f(0)^2  ) (b_1^+ )  ^2 , \nonumber \\
  P^- _{high} & = &  \left( ( b_1 ^+) _{-2}  b_{n+1} ^ + -  ( b_{n+1}  ^+) _{-2}  b_1^ +\right), \nonumber  \\
 P^- _{low} & = &  \left( ( b_1 ^-) _{-2}  b_{n+1} ^ - -  ( b_{n+1}  ^-) _{-2}  b_1^ -\right).\nonumber 
\eea

By direct calculation we have:
\begin{lemma}
 \item[(1)] $ P^- _{high} $  is a  highest weight vector  for $sl(2) \times so(8)$ of weight $(0, 2\omega_1)$. 
\item[(2)]  $P^- _{low}  \in U(so(8)). P^- _{high}$.
\end{lemma} 
\begin{proposition} \label{prop-new}We have:
\begin{itemize}
\item[(1)]  $ P^{+}  \ne 0$.
\item[(2)] $ P^+  \in \mbox{Com} ( V_{-4} (sl(2)), M_{(8)} )$; i.e.,
$ (sl(2) \otimes {\C}[t] ). P^ +=0. $
\item[(3)] $P^+ \in \widetilde V_{-2} (D_4)$ and $P ^+$ is a singular vector for $\widehat{so(8)}$ of weight $-4\Lambda_0 + 2\Lambda_1$.
 \item[(4)] $P^- _{high},  P^- _{low} \notin M_{(8)} ^{(sub)}$.
 \end{itemize}
\end{proposition}
 
 \begin{cor}
$ P^- _{high}$  is a subsingular vector  in $  M_{(8)} ^{(sub)} $  and $ P^-_{high}  + M_{(8)} ^{(sub)}  \in  M_{(8)} / M_{(8)} ^{(sub)}$  generates $V_{-4} (sl(2)) \otimes \widetilde V_{-2} (D_4) $--module isomorphic to
 $ V_{-4} (sl(2)) \otimes   \overline L_{\widehat{so(8)}}  (- 4\Lambda_0 +2 \Lambda_1) , $
where  $\overline L_{\widehat{so(8)}}  (- 4\Lambda_0 +2 \Lambda_1)$ is a highest weight $ \widetilde V_{-2} (D_4) $--module.
 \end{cor}


 \begin{proof}
We recall  the following formulas for generators for  $\widetilde V_{-2} (D_4)$:
 \bea (\mathcal E_{\epsilon_j-\epsilon_i})_{(-1)}\vac &=&\ \ :a^+_ia^-_j:+\sqrt{-1}:a^+_{n+i}a^-_j: -\sqrt{-1}:a^+_{i}a^-_{n+j}:+:a^+_{n +i}a^-_{n+j}:  \nonumber \\ && -:a^+_{j}a^-_{i}:  -\sqrt{-1}:a^+_{j}a^-_{m+i}: +\sqrt{-1}:a^+_{n+j}a^-_{i}:-:a^+_{n+j}a^-_{n+i}: \nonumber  \\
=&&  ( a^+_ i  +\sqrt{-1} a^+_{n+i} ) ( a^-_j   -  \sqrt{-1}  a ^- _{n+j} ) - (  a^+_ j  - \sqrt{-1} a^+_{n+j}     )  (  a^ - _ i  +\sqrt{-1} a^- _{n+i} ) ) \nonumber  \\
=&&  -2  ( b_{n+i}  ^-  b_{n+j }   ^+   +  b_{ j } ^+  b_{i } ^ -   ),  \nonumber 
\eea
\bea (\mathcal E_{\epsilon_i +\epsilon_j} )_{(-1)}\vac  &=& \ \ :a^+_ja^-_i:-\sqrt{-1}:a^+_{n+j}a^-_{i}:-\sqrt{-1}:a^+_{j}a^-_{n+i}:-:a^+_{n+j}a^-_{n+i}:  \nonumber \\ && -:a^+_{i}a^-_{j}:+\sqrt{-1}:a^+_{i}a^-_{n+j}:+\sqrt{-1}:a^+_{n+i}a^-_{j}:+:a^+_{n+i}a^-_{n+j}:
\nonumber  \\
= &&  ( a^+_ j   - \sqrt{-1} a^+_{n+j} ) ( a^-_ i   -  \sqrt{-1}  a ^- _{n+i } ) -  (  a^+_ i    - \sqrt{-1} a^+_{n+i}     )  (  a^ - _ j    - \sqrt{-1} a^- _{n+j } ) ) \nonumber  \\
= && - 2 \sqrt{-1} (b_{j} ^ + b_{n+i } ^ +  - b_{i } ^+  b_{n+j} ^+    ) .\nonumber  
\eea
 The generators of $V_{-4}(sl(2))$ can be expressed as follows:
 \bea
e& =&  \sqrt{-1} \sum_{i=1} ^n b_i ^+ b_{n+i} ^-, \nonumber \\
f& =& - \sqrt{-1} \sum_{i=1} ^n b_i ^- b_{n+i} ^+, \nonumber \\
h &=& - \sum_{i=1} ^n  ( b_i ^+  b_i ^- -b_{n+i} ^+ b_{n+i} ^- ).\nonumber 
\eea
By direct calculation we get
\bea
 P^+  &=&   f(-1) ( (b_1 ^ + )_{(-1)} ) ^ 2 {\vac}  +  e(-1) ( (b_5 ^ + )_{(-1)} ) ^ 2 {\vac} -   \sqrt{-1} h(-1) ( (b_1 ^ + )_{(-1)} ) b_5^ +  \nonumber \\
    &=&  - 3 \sqrt{-1}  (  ( b_1 ^+) _{-2}  b_5^ +  - ( b_5 ^+) _{-2}  b_1^ + )     -  \sqrt{-1} \ \sum_{i=2} ^ 4           (       b_{ 1 } ^+  b_{i } ^ -  + b_{5 }   ^+  b_{4 +i}  ^-     )    (       b_{ 1 } ^+  b_{ 4+ i  } ^ +   - b_{5 }   ^+  b_{i }  ^+     ) \nonumber \\
 &=& - \frac{1}{4} \sum_{i=2} ^{4} ({\mathcal E}_{\varepsilon _1-\varepsilon_i} )_{(-1)} ({\mathcal E}_{\varepsilon _1+\varepsilon  _i} )_{(-1)}  \vac \in \widetilde  V_{-2} (D_4).  \label{formula-D4} \eea
 Note that  (\ref{formula-D4})  gives a non-trivial projection of  the singular vector in $V^{-2} (D_4)$  from \cite[Theorem 3.1]{P-Glasnik} in the case $n=1$, $\ell =4$.

 Assume that $P^- _{low}  \in M_{(m)} ^{(sub)} $. Since it has conformal  weight $2$ we should have
 $$ P^- _{low}  = P^-   { (0)} + P^- (1) + P ^- (2),$$
 where 
 \bea
 P^- (0) & \in &V_{-4} (sl(2)) \otimes  \widetilde V_{-4} (D_4)   \nonumber \\
  P^- (1) & \in & \widetilde L_{\widehat{sl(2)}}  (- 6 \Lambda_0 +2 \Lambda_1) \otimes \widetilde L_{\widehat{so(8)}}  (- 4\Lambda_0 +2 \Lambda_1),   \nonumber \\
    P^- (2) & \in &  \widetilde L_{\widehat{sl(2)}}  (- 8 \Lambda_0 +4 \Lambda_1) \otimes \widetilde L_{\widehat{so(8)}}  (- 6\Lambda_0 +4 \Lambda_1).   \nonumber \eea

  By using fusion rules and the  fact that $P ^+$ is singular vector in  $\widetilde V_{-4} (D_4) $ one easily sees that 
 $ P_{(3)} ^+ P^- (i)=  0$ for $i=0,1,2$. But 
 $$P_{(3)} ^+ P^- _{low} = \nu \vac  \quad \nu   \ne 0, $$
 a contradiction. 
This  proves that $P ^- _{low} \notin M_{(m)} ^{(sub)} $. \end{proof}

 Note that $P^- _{high} $ is subsingular for  $\widehat{sl(2)}$:  
 \bea
 e(0) P^- _{high} &=&  0, \nonumber \\
 e(1) P^- _{high} &=& (b_1 ^+ ) ^2,  \nonumber \\
 e(2) P ^- _{high}&= & 0,\nonumber \\
  f(0) P^- _{high} &=&  0, \nonumber \\
 f(1) P^-_{high} &=& -  (b_5 ^+ ) ^2,  \nonumber \\
 f(2) P ^- _{high} &= & 0.\nonumber 
 \eea
 So $P^-_{high},  P^-_{low}\notin \mbox{Com} (V_{-4}(sl(2)), M_{(8)})$.

  Finally we conclude this section with one observation which gives an argument why our subsingular vectors do not appear in the classical Howe setting.
 
\begin{rem} \label{observation}
By taking suitable conformal vector, one can realize the Zhu's algebra of  the Weyl vertex algebra $M_{(m)}$ as  the classical Weyl algebra $A_m$ with generators
$x_i$, $\frac{\partial}{\partial x_i }$ and commutation relation $[ \frac{\partial}{\partial x_i }, x_j] = \delta_{i,j}$.
It is easy to see that the Zhu's functor maps subsingular vector   $P^- _{high}$ to zero. In our opinion this explain why subsingular vectors which do appear in our analysis, do not appear in classical settings.
\end{rem}

 \subsection{An application to conformal embeddings  }
 Here we want to study conformal embedding  $A_3 \times Z \hookrightarrow  D_4$, $Z$ $1$--dimensional abelian,  at level $k$,  and prove that the subalgebra of $V_{-2} (D_4)$ generated by $A_3$ is simple.
 
 Note first  that in previous sections appeared two vertex algebras associated to $A_3$: $\overline{V}_{-2} (A_3) $ and $ \widetilde V_{-2} (D_3)= \widetilde V_{-2} (A_3)$. Recall that 
 $\overline{V}_{-2} (A_3)  $ is  the vertex subalgebra  of $ V_{-1} ( sl(8) ) \subset V_{-1/2}(C_8)$ generated by 
the elements  of $sl(4) $ in the conformal embedding  $sl(2) \times sl(4) \hookrightarrow  sl(8) $ at $k=-1/2$. On the other hand  $\widetilde V_{-2} (D_3)$ is a subalgebra of  $ V_{-1/2}(C_6)$ generated by the elements of $so(6)$  in the conformal embedding  $sl(2) \times so (6) \hookrightarrow  sp(12) $ at $k=-1/2$.

  Since  $\widetilde V_{-2} (D_3)$ is simple by Theorem  \ref{thm-simplicity}  and  $\overline{V}_{-2} (A_3) $ is not simple by Example  \ref{ex-non-simp} we conclude that  $\overline{V}_{-2} (A_3) \ne   V_{-2} (A_3)$.

  \begin{lemma}
The vertex subalgebra of $V_{-2} (D_4)$ generated by the elements in $A_3$ is simple.
 \end{lemma}
 \begin{proof}
By direct calculation one sees that
 \bea && w= (\mathcal E_{-\epsilon_1-\epsilon_4})_{(0)} P^+ \in  \overline{V}_{-2} (A_3) \label{max-a3} \eea
 is a non-trivial projection of the singular vector in $V^{-2} (A_3)$ from \cite[Theorem 8.2]{ArM-II}.  
 Since vector $w$  generates the maximal ideal of $\overline{V}_{-2} (A_3)$, now relation (\ref{max-a3}) shows that $w$ belongs to the maximal ideal in  $\widetilde V_{-2} (D_4)$, and hence we get a non-vanishing homomorphism
 $V_{-2} (A_3) \rightarrow V_{-2}(D_4)$. The claim follows.
 \end{proof}

Thus we conclude:
 \begin{cor} \label{cor-non-simp-2} We have:
 \begin{itemize}
\item[(1)] $\overline{V}_{-2} (A_3) \otimes M(1) $ is  conformally embedded into $\widetilde V_{-2} (D_4)$, but it is not embedded into $V_{-2}(D_4)$.
\item[(2)]  $ {V}_{-2} (A_3) \otimes M(1) $ is  conformally embedded  into $V_{-2} (D_4)$.  
\end{itemize}
\end{cor}

\vskip3pt
 \footnotesize{
  \noindent{\bf D.A.}:  Department of Mathematics, Faculty of Science, University of Zagreb, Bijeni\v{c}ka 30, 10 000 Zagreb, Croatia;\newline
{\tt adamovic@math.hr}
  
\noindent{\bf V.K.}: Department of Mathematics, MIT, 77
Mass. Ave, Cambridge, MA 02139;
{\tt kac@math.mit.edu}

\noindent{\bf P.MF.}: Politecnico di Milano, Polo regionale di Como,
Via Valleggio 11, 22100 Como,\newline
Italy; {\tt pierluigi.moseneder@polimi.it}

\noindent{\bf P.P.}: Dipartimento di Matematica, Sapienza Universit\`a di Roma, P.le A. Moro 2,
00185, Roma, Italy;\newline {\tt papi@mat.uniroma1.it}

\noindent{\bf O.P.}:  Department of Mathematics, Faculty of Science, University of Zagreb, Bijeni\v{c}ka 30, 10 000 Zagreb, Croatia;\newline
{\tt perse@math.hr}
}

\end{document}